\documentclass{amsart}
\usepackage{fancyhdr}
\usepackage{amsthm}
\usepackage{amssymb}
\usepackage{amscd}
\usepackage{amsmath, pb-diagram}
\usepackage[all]{xy}
\usepackage{setspace}
\usepackage{graphics,epsfig}
\usepackage{graphicx}
\usepackage{float}
\usepackage{array}
\usepackage{xfrac}
\usepackage{faktor}
\usepackage{multicol}
\usepackage[mathscr]{euscript}
\usepackage{mathtools}
\usepackage{tikz-cd}
\usepackage[margin=2.9cm]{geometry}
\usepackage{url}
\usepackage{nicefrac}
\usepackage{xfrac}
\numberwithin{equation}{section}

\usepackage{enumerate}
\usepackage[hang,flushmargin]{footmisc}
\usepackage{amsmath}
\usepackage{amsfonts}
\usepackage{amssymb}
\usepackage{amsthm}
\usepackage{color}
\usepackage{blindtext}
\usepackage{tikz}
\usepackage{tikz-cd}
\usepackage{xcolor}
\usepackage{graphicx}
\usepackage{setspace}
\usepackage{float}
\usepackage{multicol}
\graphicspath{ {./images/} }
\usepackage{fullpage}
\usepackage{hyperref}
\usepackage{array}
\newcolumntype{C}[1]{>{\centering\arraybackslash}m{#1}}
\usepackage[justification=centering]{caption}

\allowdisplaybreaks

\date{\today}

\theoremstyle{plain}

\theoremstyle{definition}

\theoremstyle{plain}
\newtheorem{theorem}{Theorem}[section]
\newtheorem{lemma}[theorem]{Lemma}
\newtheorem{corollary}[theorem]{Corollary}

\theoremstyle{definition}
\newtheorem{definition}[theorem]{Definition}

\newtheorem{?}[theorem]{Question}
\newtheorem{example}[theorem]{Example}

\title{Representations of integers as quotients of sums of distinct powers of three}

\author[K. Anders]{Katie Anders}
\address{University of Texas at Tyler}
\email{kanders@uttyler.edu}

\author[M. Dawsey]{Madeline Locus Dawsey}
\address{University of Texas at Tyler}
\email{mdawsey@uttyler.edu}

\author[B. Reznick]{Bruce Reznick}
\address{University of Illinois at Urbana-Champaign}
\email{reznick@illinois.edu}

\author[S. Sisneros-Thiry]{Simone Sisneros-Thiry}
\address{California State University- East Bay}
\email{simone.sisnerosthiry@csueastbay.edu}


\begin{document}
\begin{abstract}
   Which integers can be written as a quotient of sums of distinct powers of three?  We outline our first steps toward an answer to this question, beginning with a necessary and almost sufficient  condition.  Then we discuss an algorithm that indicates whether it is possible to represent a given integer as a quotient of sums of
    distinct powers of three.  When the given integer is representable, this same algorithm generates all possible representations.  We develop a categorization of representations based on their connections to $0,1$-polynomials and give a complete description of the types of representations for all integers up to 364.  Finally, we discuss in detail the representations of 7, 22, 34, 64, 
    and 100, as well as some infinite families of integers.\\
\end{abstract}

\maketitle


\section{Introduction}\label{intro}

We investigate which integers can be written as a quotient of sums of distinct powers of 3. Such an integer $m$ is of the form 
\begin{equation}\label{1.1}
m=\frac{3^{a_1} + \cdots + 3^{a_{\ell-1}}+3^{a_\ell}}{3^{b_1} +\cdots  + 3^{b_{n-1}} + 3^{b_n}}
= 3^{a_\ell-b_n}\cdot
\frac{3^{a_1-a_\ell} + \cdots + 3^{a_{\ell-1} - a_\ell} + 3^{0}}{3^{b_1-b_n} + \cdots + 3^{b_{n-1} - b_n} +3^{0}},
\end{equation}
where $a_1, a_2,\ldots, a_\ell$ and $b_1, b_2, \ldots, b_n$ are non-negative integers with $a_1 > a_2>\cdots > a_\ell$, $b_1 >b_2> \cdots > b_n$, and $a_\ell-b_n\geq0$.  We can restrict our attention to integers $m$ congruent to $1\bmod 3$.  Why?  Since the numerator and denominator of the right side of \eqref{1.1} are congruent to $1\bmod 3$, we have $m\equiv1\pmod 3$ if $a_\ell-b_n=0$ and $m\equiv0\pmod 3$ if $a_\ell-b_n>0$. 
If $m\equiv0\pmod3$, then we can write $m=3^km_1$, where $m_1\equiv1\pmod3$.  Consequently, we will only consider integers congruent to $1\bmod 3$.

A straightforward estimate, given below as Theorem \ref{T:feasible}, shows
that any integer written as a quotient of sums of distinct powers of three must be contained in
\begin{equation}\label{feasible}
\bigcup_{r=0}^\infty \bigl( \tfrac 23\cdot3^r, \tfrac 32 \cdot 3^r \bigr).
\end{equation}
Can every integer congruent to $1\bmod 3$ belonging to the union $\eqref{feasible}$ be written as a quotient of sums of distinct powers of 3?  No, this necessary condition is nearly sufficient but not quite. The smallest integers $m\equiv 1 \pmod 3$ that belong to the union of intervals in \eqref{feasible} but have no representation of the form \eqref{1.1} are
$m=529$, $m=592$, $m=601$, and $m=616$, and there are only ten exceptions less than $6.2 \times 10^6$. 

 A very preliminary version of some of the material in this paper was presented by the third author in  2015 \cite{R}, and this presentation
included the next smallest counterexample,
5368, which was found by his student Sakulbuth Ekvittayaniphon.
Independently, a few years later, Jeffrey Shallit and his student Sajed 
Haque \cite{H} found a total of eleven exceptions less than 107 million. See also A339636 in \cite{OEIS}, in
which it is stated that ``A simple automaton-based (or breadth-first search) algorithm can establish in $O(n)$ time whether 
$n$ is \ldots,'' in our terms, representable. These exceptions are also discussed in \cite[p,4]{BMRS}; the
distinction of local and universal representations does not arise. 
See also the thesis of the fourth author \cite[pp.128-130]{SST}.

\subsection{Definitions, Notation, and Background Material}\label{DefinitionsNotationBackgroundSection} For $r \ge 1$, let
\begin{equation*}\label{defA}
\mathcal A_r = \left \{ 1 + \sum_{k=1}^{r-1} \epsilon_k 3^k + 3^r\ \bigg| \ \epsilon_k \in \{0,1\} \right \}\quad \text{ and }
 \quad \mathcal A = \{1\} \cup \bigcup_{r=1}^\infty \mathcal A_r.
\end{equation*}
Returning to \eqref{1.1}, we see that in the rightmost fraction, both the numerator and denominator are elements of $\mathcal{A}$.  We use the following notation for the set of 0,1-polynomials with constant term 1, also known as {\it Newman polynomials}, which are investigated more thoroughly in \cite{OP}:
\begin{equation*}\label{defP}
\mathcal P_r = \left \{ 1 + \sum_{k=1}^{r-1} \epsilon_k x^k + x^r \ \bigg| \ \epsilon_k \in \{0,1\} \right \}\quad \text{ and }
 \quad \mathcal P = \{1\} \cup \bigcup_{r=1}^\infty \mathcal P_r.
\end{equation*}
Note that $m \in \mathcal A_r$ if and only if there exists a polynomial $p(x) \in \mathcal P_r$ such that $m = p(3)$. With this notation,
the subject of our investigation is $\mathcal A / \mathcal A$, defined as the set of 
integers congruent to $1 \bmod 3$ that can be written as a quotient of two elements of $\mathcal A$. 
Since $1 \in \mathcal A$, we see that 1 can be the denominator of such a quotient, and it follows that $\mathcal A \subseteq \mathcal A / \mathcal A$. There are, however, elements of $\mathcal A / \mathcal A$ that do not belong to $\mathcal{A}$.  One such example is 7, which we shall consider shortly.

We use the notation $w=[\epsilon_r\;\epsilon_{r-1}\;\dots\;\epsilon_1\;\epsilon_0]_3$ to indicate the standard base 3 representation of the integer $w$, where the digits $\epsilon_r,\dots,\epsilon_0\in\{0,1,2\}$.  For example, the standard base 3 representation of the integer 7 is $7=[21]_3$.  We see from this standard base 3 representation that 7 does not belong to $\mathcal{A}$; because
\begin{equation}\label{reps of 7}
7 = \frac{28}{4} = \frac{[1001]_3}{[11]_3} = \frac{3^3+1}{3+1} = \frac{p(3)}{q(3)},
\end{equation}
where $p(x) = x^3 + 1 \in \mathcal P_3$ and $q(x) = x + 1 \in \mathcal P_1$, we have that $7\in\mathcal{A}/\mathcal{A}$. 

Throughout this document, we say an integer is \textit{representable} if it can be written as a quotient of sums of distinct powers of 3 using any of the types of formats illustrated in \eqref{reps of 7}. In the example of the integer 7, we notice that $\frac{p(x)}{q(x)} = x^2 - x + 1$, which, when evaluated at any integer $x$, will yield an integer. 
We call such a representation {\it universal}.
Not every representation is universal; those that are not we call \emph{local}, and the smallest example of a local representation in base 3 is 
\[
22 = \frac{22\cdot37}{37} = \frac{814}{37} = \frac{3^6 + 3^4 + 3 + 1}{3^3 + 3^2 + 1}=\frac{p(3)}{q(3)},
\]
where $p(x)=x^6 + x^4 + x + 1\in\mathcal{P}_6$ and $q(x)=x^3 + x^2 + 1\in\mathcal{P}_3$.  Note that $q(x)$ does not divide $p(x)$ in this example, and so in general the quotient $\frac{p(b)}{q(b)}$ need not be an integer when evaluated at an integer $b \neq 3$. As an example, $\frac{p(1)}{q(1)}= \frac 43$. Note also that $22 =
3^3-3^2 + 3 + 1$, and it is not surprising that
\[
x^6 + x^4 + x + 1 = (x^3 + x^2+1)(x^3 -x^2 +x+1) + x^3(x-3),
\]
since the polynomials involved in the representation of $22$ indicate that the remainder from polynomial division must be divisible by $x-3$.  We now state the definitions of \textit{universal representation} and \textit{local representation} more formally.

\begin{definition}
Let $m$ be an integer such that $m\equiv 1\pmod 3$ and $m=\frac{p(3)}{q(3)}$, where $p(x), q(x)\in \mathcal{P}$.  If $q(x)\mid p(x)$, then the representation of $m$ as $\frac{p(3)}{q(3)}$ is \textit{universal}.  If $q(x)\nmid p(x)$, then this representation of $m$ is \textit{local}.
\end{definition}

Notice that any integer $m\equiv1\pmod{3}$ which can be written as a sum of distinct powers of 3 has a universal representation $m=\frac{p(3)}{q(3)}$ with $q(x)=1$.  Throughout, we refer to a universal representation for which $q(x)=1$ as a \emph{trivial} universal representation.

Next we prove the aforementioned theorem that states that an integer written as a quotient of sums of distinct powers of 3 must belong to the union given in \eqref{feasible}.
\begin{theorem}\label{T:feasible}
For any integer $r$, let $\mathcal I_r = \left( \frac 23\cdot3^r, \frac 32 \cdot 3^r \right)$.  If $m \in \mathcal A/\mathcal A$, then $m \in \mathcal I_r$ for some $r$.  Moreover, if $m = \frac{p(3)}{q(3)}$ for some
$p(x),q(x) \in \mathcal P$, then there exists an integer $e$ such that
$p(x) \in \mathcal P_{e+r}$ and $q(x) \in \mathcal P_{e}$. 
\end{theorem}
\begin{proof}
Suppose $p(x) \in \mathcal P_d$.  Then
\[
3^d < 1 + 3^d  \le p(3) \le 1 + \sum_{k=1}^{d-1} 3^k + 3^d = \frac{3\cdot 3^d - 1}{3-1} < \frac 32 \cdot 3^d.
\]
Accordingly, if $q(x) \in \mathcal P_e$, then $3^e < q(3) < \frac{3}{2}\cdot 3^e$, so
\[
\frac{2}{3}\cdot 3^{d-e}=\frac{3^d}{\frac 32 \cdot 3^e} < \frac {p(3)}{q(3)} 
< \frac{\frac 32 \cdot 3^d}{3^e}=  \frac 32 \cdot 3^{d-e},
\]
and thus $\frac{p(3)}{q(3)} \in \mathcal{I}_{d-e}$.
This inequality holds regardless of whether $\frac {p(3)}{q(3)}$ is an integer. If $m = \frac {p(3)}{q(3)}$, then $m \in \mathcal I_{d-e}$,
establishing the necessary condition; if $m \in \mathcal I_r$, then $d = e + r$.
\end{proof}

What do the intervals in Theorem \ref{T:feasible} remind us of? This question led to our initial interest in these representations of integers.  In 2019, Athreya, Tyson, and the third author studied the
standard Cantor set,
\[
C = \left\{ \sum_{k=1}^\infty \frac{\alpha_k}{3^k} \ \bigg \vert \ \alpha_k \in \{0,2\} \right\},
\]
and they proved in \cite[Theorem 1]{ART} that 
\begin{equation}\label{ART union}
\left\{\frac uv \ \bigg\vert\  u, v \in C, v\neq 0\right\} = \bigcup_{r= - \infty}^\infty \left[ \frac23\cdot 
3^r, \frac 32 \cdot 3^r\right].
\end{equation}
Note the similarity between this union and that given in $\eqref{feasible}$.  We also highlight and explain some of the differences.  Since we are only considering integers in our work, the union in \eqref{feasible} is over $r \ge 0$. Another difference is that we only consider integers congruent to $1\bmod 3$, while in \eqref{ART union}, there is no restriction regarding the congruence class of the integer mod 3. For example, 
$\frac 13, \frac 23 \in C$, so $\frac{2/3}{1/3}=2$ is a quotient of elements of the Cantor set, but of course $2\equiv2\pmod{3}$ cannot be represented as a quotient of sums of distinct powers of 3, as we have discussed. 

The usual construction of the Cantor set involves a sequence of sets $C_n$ consisting of $2^n$ closed intervals, each of length $3^{-n}$. Consider
the set of quotients of left endpoints of these intervals:
\begin{align*}
 \left\{ \frac{\sum_{k=1}^d \frac{\alpha_k}{3^k}}{\sum_{j=1}^e \frac{\beta_j}{3^j}} \ \bigg\vert \ \alpha_k, \beta_j \in \{0,2\} \right\} 
&= \left\{\frac{ 2\cdot 3^{-d}\sum_{k=1}^d \frac 12\alpha_k\cdot3^{d-k}}
{2\cdot 3^{-e} \sum_{j=1}^e \frac 12\beta_j\cdot3^{e-j}} \ \bigg\vert \ \alpha_k, \beta_j \in \{0,2\} \right\} \\
&= \left\{ \frac{ 3^{e-d} \cdot\sum_{k=1}^d \frac 12\alpha_k\cdot3^{d-k}}
{\sum_{j=1}^e \frac 12\beta_j\cdot3^{e-j}} \ \bigg\vert \ \alpha_k, \beta_j \in \{0,2\} \right\}.
\end{align*}
Each element of this set is a quotient of two elements of $\mathcal P$, evaluated at $x=3$, multiplied by a power of $3$. 

The question of which polynomials are quotients of two elements of $\mathcal P$ seems to be both open and hard. One observation is that if $r(x) = \frac{p(x)}{q(x)}$,
then the zeros of $r(x)$ are a subset of the zeros of $p(x)$. Determining the zeros of polynomials in $\mathcal P$ also seems to be a very hard question; see \cite{OP}.

In 1987, John Loxton and Alf van der Poorten wrote a paper entitled ``An awful problem 
about integers in base four" \cite{LvdP}, which was dedicated to proving that every odd integer can be written
 as a quotient of sums and differences of powers of 4. What was ``awful" was only that there was not
  an elegant top-down solution.  Their paper reports that Selfridge and LaCampagne asked a similar question for 3, and the
   authors mentioned that if differences are removed, as in the case we are investigating,
``There is some difficulty in describing which integers can occur."

We concur with this quote, and in this paper we report our findings and early progress in determining which integers congruent to $1\bmod 3$ can be written as quotients of sums of distinct powers of 3.  In Section \ref{special families}, we exhibit several families of integers having only universal representations, and then we consider $100$, the smallest integer having only universal representations but not belonging to any of these families.  In contrast to the integers discussed in Section \ref{special families}, some integers have only local representations.  We investigate two such examples, $22$ and $34$, in Section \ref{OnlyLocalSection}.  Then, in Section \ref{64Section}, we present $64$ as an example of an integer having both nontrivial universal representations and local representations.  In Section \ref{catalog}, we provide a table that indicates for each integer $m\equiv1\pmod{3}$ contained in the union \eqref{feasible} up to $m=364$ whether $m$ has a universal representation and whether $m$ has a local representation.  Finally, we outline further directions to explore in Section \ref{FutureWork}.

The following examples illustrate a few of the possible combinations of types of representations that can exist for a single integer $m\equiv1\pmod{3}$.

\begin{example}\label{31 example}
The integer $31=3^3+3+1$ has a trivial universal representation with $p(x) = x^3 + x + 1$
and $q(x) = 1$. It is also the case that if $g(x) = x^5 - x^3 + 1$ and $h(x) =
x^2 - x + 1$, then $\frac{g(3)}{h(3)}=\frac{217}{7}=31$, and $h(x)\nmid g(x)$.  Since neither $g(x)$ nor $h(x)$ is in $\mathcal P$, we know that $\frac{g(3)}{h(3)}$ is not a universal representation of $31$. However, with a suitable common multiplier for $g(x)$ and $h(x)$, we can obtain polynomials in $\mathcal{P}$ and see that $31$ has a local representation in addition to its trivial universal representation.  In fact, $31$ actually has many local representations, and the algorithm in Section \ref{sec:algorithm} describes how to find them. As a specific example of such a local representation, observe that
\[
\frac{x^5 - x^3 + 1}{x^2-x+1} = \frac{(x^5-x^3+1)(x^4 + x^3+x^2+x+1)}
{(x^2-x+1)(x^4 + x^3+x^2+x+1)} = \frac{x^9 + x^8 + x^2 + x + 1}{x^6 + x^4+x^3+x^2+1}.
\]
\end{example}

\begin{example}\label{37 example}
As another example of an integer that has both a trivial universal representation and a local representation, consider
$37 = 3^3 + 3^2 + 1$.   We can also write
\[
37 =  \frac{3^6 - 3^3 + 1}{3^3 - 3^2 + 1}.
\]
The polynomials $x^6-x^3+1$ and $x^3-x^2+1$ are not in $\mathcal{P}$, but if we multiply each by $x^5+x^4+x^3+x^2+x+1$,
we obtain polynomials that do belong to $\mathcal{P}$, and we have $37 =\frac{p(3)}{q(3)}$,
where $p(x)=x^{11} + x^{10} + x^9 + x^2 + x + 1$ and $q(x)=x^8 + x^5+x^4+x^3+x+1$.  To conclude that this is a local representation, we must have that $q(x)\nmid p(x)$.  We could show this with polynomial long division, or we can consider evaluating the quotient $\frac{p(x)}{q(x)}$ at $x=2$, for example.  This yields $\frac{3591}{315}$, which is not an integer and tells us the representation of $37$ as $\frac{p(3)}{q(3)}$ must be local.
\end{example}

\begin{example}\label{841 example}
It is also possible for an integer to have two universal representations with different polynomial quotients.  For example,
$841 = 3^6 + 3^4 + 3^3 + 3 + 1$ is a trivial universal representation of $841$ with the polynomial quotient being $x^6 + x^4 + x^3 +x + 1$. In addition, $841$ has a nontrivial universal representation with polynomial quotient $x^6+x^5-x^4-2x^3+x+1$, since $841=3^6 + 3^5 - 3^4 - 2\cdot 3^3 + 3 + 1=
\frac{p(3)}{q(3)}$ with $p(x)=x^{15} + x^{14} + x^{11} + x^5 + x^4 + x^2 + x + 1$ and $q(x)=x^9 + x^7 + x^6 + x^5 + x^4 + x^3 + x^2 + 1$. 
\end{example}

For a given integer $m$, we have developed an algorithm for determining whether $m$ is an element of $\mathcal{A}/\mathcal{A}$. This algorithm is a refinement of a multiplication transducer, which we shall discuss in more detail in Section \ref{sec:algorithm}.  As noted previously, we need only consider the integers $m \equiv 1 \pmod{3}$ belonging to an interval of the form given in Theorem \ref{T:feasible}.  For such $m$, this algorithm explores all possibilities for multiples of $m$ by $q$, a sum of distinct powers of 3, such that $mq = p$ is also a sum of distinct powers of 3. This can be done manually or through the construction of a directed graph that represents all possible multiplications.  The algorithm leverages intuition similar to more familiar multiplication algorithms in base 10, performing multiplication in steps, starting with the ones digit, and ``carrying" powers of 3. However, our algorithm does not require the summation of partial products. 

 Rather than using our algorithm to compute the product $p$ of given integers $m$ and $q$, we instead start with a given integer $m\equiv1\pmod3$ and construct both $p$ and $q$ simultaneously to ensure that they are elements of $\mathcal{A}$, choosing only paths forward that preserve this necessary condition. Thus, we begin with $3^0=1$ copy of $m$, and then for each positive power $i$ of 3, we need to choose whether or not to include $3^i$ additional copies of $m$ in $q$, building the base 3 representation of $q$ as we move through the algorithm.

\subsection{The Algorithm, Digraphs, and Multiplication Transducers}\label{sec:algorithm}

We begin with a familiar example, using the algorithm to construct the quotient given in \eqref{reps of 7} for $m=7$.  Denote this quotient by $7=\frac{p}{q}\in\mathcal{A}/\mathcal{A}$.  The procedure of constructing this quotient involves iterating through steps, where each step ``records" the $i$th digit of the product $7\cdot q=p$ in base 3, as well as the $i$th digit of $q$.  Since $p,q\in\mathcal{A}$, it is necessary to record a digit of 1 in the $3^0$ place of the base 3 expansions of both $p$ and $q$.  Write the base 3 expansion of $7$ as $7 = (2\cdot3^1+1)\cdot3^0$.  After recording 1 into the $3^0$ place of $p$, the remaining unrecorded value is $(2\cdot3^1)\cdot3^0=2\cdot3^1$ to ``carry" into the next step, which is the $3^1$ step.  Note that the number of copies of $3^1$ in this unrecorded value that we must carry into the next step is $2$, which is congruent to $2\bmod{3}$.  In order to record either $0$ or $1$ into the $3^1$ place of $p=7\cdot q$, we need to first add more copies of $3^1$ to $q$. Hence we add $m=7$ to  the carried value of $2$, which  gives us  $ (m+2)\cdot3^1  =  (7+2)\cdot3^1= 9\cdot3^1$.  To account for this addition of $m\cdot3^1=7\cdot3^1$ into $7\cdot q=p$, we must record a digit of 1 in the $3^1$ place of $q$.  Now, we observe that $9\cdot3^1 = 3\cdot3^2$, so $3$ is the value which we must carry at the $3^2$ step.  Since 3 is a sum of distinct powers of 3, we can stop iterating and complete the construction.  In fact, we have that $3\cdot3^2=3^3$, so we record a digit of 0 in the $3^2$ place of $p$ and a digit of 1 in the $3^3$ place of $p$.  The result of the algorithm is therefore $7\cdot q=p$, which is given by the equation $7\cdot(3^1 + 3^0) = 3^3 + 3^0$.  It follows that $7=  \frac{p}{q} = \frac{[1001]_3}{[11]_3}$, the representation from \eqref{reps of 7}. A complete description of all possible representations of 7 is given by Theorem \ref{T:onlyuniversal}  Part (3) with $r = k=1$. 

In general, we can perform the algorithm iteratively by following the steps below, and in the case that $m = \frac{p}{q} \in \mathcal{A}/ \mathcal{A}$, each step $i$ produces the digit of the $i$th power of 3 in the base 3 expansions of both $q$ and $p$. We record either $0\cdot3^i$ or $1\cdot3^i$ into the base 3 expansions of $q$ and $p$ in the $i$th step, depending on the result of the previous step, and then we consider the remaining unrecorded value as carrying multiples of $3^{i+1}$ into the next step.  We refer to the number of copies of $3^{i+1}$ which are carried into the next step as the \textit{carry value}.  As described in the following procedure, we use an initial carry value of 0 in Step 0.  The general algorithm is that for Step $i$, where $i\geq1$:
\begin{enumerate}
\item Determine the number of copies of $3^i$ recorded into the base 3 expansions of $p$ and $q$ in Step $i$ by following the bullet point below which corresponds to the carry value at the end of Step $i-1$.
\begin{itemize}
    \item If the carry value is congruent to $1 \bmod 3$: record $0\cdot3^i$ into $q$, record $1\cdot3^i$ into $p$, and subtract $1$ from the carry value.
    \item If the carry value is congruent to $2 \bmod 3$: add $m$ to the carry value, record $1\cdot3^i$ into $q$, and record $0\cdot3^i$ into $p$.
    \item If the carry value is congruent to $0 \bmod 3$, we have a choice: either
    \begin{itemize}
        \item add $m$ to the carry value, record $1\cdot3^i$ into $q$, record $1\cdot3^i$ into $p$, and subtract $1$ from the carry value (note: we always make this choice at Step $0$), or
        \item record $0\cdot3^i$ into $q$, and record $0\cdot3^i$ into $p$.  
    \end{itemize}
\end{itemize}
\item After recording $0\cdot3^i$ or $1\cdot3^i$ into $p$ and $q$ according to the conditions above and subtracting 1 from the carry value if necessary, the result will be divisible by $3^{i+1}$.  Factor out $3^{i+1}$ from this result, and what remains is the new carry value at the end of Step $i$.
\item Move to Step $i+1$ using this new carry value.
\end{enumerate}

The algorithm terminates when either (i) the carry value is $w$, a sum of distinct powers of 3, or  (ii) all possible destinations from a given Step $i$ have already been visited without arriving at a sum of distinct powers of 3. 

In Case (i), the base 3 expansion of the integer $q$ is complete when the algorithm terminates, and the base 3 expansion of the integer $p$ is completed by appending the base 3 representation of $w$ to the left of the values already recorded for $p$. In our example $m=7$, the algorithm terminated when we arrived at a carry value of 3 from $9*3^1 = 3*3^2$.  The algorithm terminated because the carry value $w=3$ is a sum of distinct powers of 3. We then completed $p$ by appending $[{\bf10}]_3$, the base 3 expansion of $3$, to arrive at $7 = \frac{[{\bf10}01]_3}{[11]_3}$. 

In Case (ii), the integer $m$ is not representable, because the algorithm enters an infinite loop and never arrives at a finite quotient. This is precisely what happens with $m=529$, $m=592$, $m=601$, $m=616$, and the other integers that are congruent to $1\bmod 3$ and belong to the union given in \eqref{feasible} but cannot be represented as quotients of sums of distinct powers of 3.

The remarks below provide additional justification as to why the algorithm builds a representation of $m$ as $\frac{p}{q}$, where $p,q\in\mathcal{A}$.

\begin{itemize}

\item The procedure builds the base 3 expansions of $q$ and $p$ one digit at a time based on the powers of 3 required to obtain $mq=p$, where $p$ and $q$ are both sums of distinct powers of 3.  Since $p$ and $q$ cannot have any digits of 2 in their base 3 expansions, we must add $m$ to any carry values congruent to $2\bmod3$ in order to record either $0\cdot3^i$ or $1\cdot3^i$ for $p$ in Step $i$.  This addition corresponds to adding $1\cdot3^i$ to $q$, since then we will have added $m\cdot3^i$ to the product $mq=p$.  After adding $m\equiv1\pmod{3}$ to such a carry value in Part (1) of the algorithm, the result will be divisible by $3^{i+1}$, so we may move to Part (2) of the algorithm.  

\item The carry value is the part of the remaining unrecorded value, after recording $0\cdot3^i$ or $1\cdot3^i$ into $p$ in Step $i$, which carries multiples of $3^{i+1}$ into the next step.  For this reason, we must subtract $1$ from the carry value every time we record $1\cdot3^i$ into $p$, so that what remains is the part of the carry value which was unrecorded in Step $i$.

\item When a carry value is congruent to $0\bmod3$, the two choices in Part (1) of the algorithm correspond to the two different ways of arriving at a multiple of $3^{i+1}$ to be carried from Step $i$ into Step $i+1$.  The first way is to add $m$ to the carry value and correspondingly record $1\cdot3^i$ into $q$, then subtract $1\cdot3^i$ from the carry value after recording a digit of $1$ into the $3^i$ place of $p$, and move to Part (2) of the algorithm with the resulting multiple of $3^{i+1}$.  The second way is to simply factor out another 3 from the carry value, which is divisible by 3 in this case, to get a multiple of $3^{i+1}$.

\end{itemize}

The algorithm can be illustrated through a directed graph with the carry values as vertices. We present the graph for $m=22$ in Figure \ref{22digraph}.  In Example \ref{22 example}, we discuss how to use the directed graph to obtain representations of $22$ as a quotient. 

\begin{example}\label{22 example}  Let's consider the algorithm for $m=22$.  We will use Figure \ref{22digraph} to keep track of what happens in each step.  

\begin{center}
\begin{figure}[h!]    

\tikzset{every picture/.style={line width=0.75pt}} 

\begin{tikzpicture}[x=0.75pt,y=0.75pt,yscale=-1,xscale=1]

\draw   (54,81.5) .. controls (54,71.84) and (61.84,64) .. (71.5,64) .. controls (81.16,64) and (89,71.84) .. (89,81.5) .. controls (89,91.16) and (81.16,99) .. (71.5,99) .. controls (61.84,99) and (54,91.16) .. (54,81.5) -- cycle ;
\draw    (102.5,81) -- (137.5,81) ;
\draw [shift={(139.5,81)}, rotate = 180] [color={rgb, 255:red, 0; green, 0; blue, 0 }  ][line width=0.75]    (10.93,-3.29) .. controls (6.95,-1.4) and (3.31,-0.3) .. (0,0) .. controls (3.31,0.3) and (6.95,1.4) .. (10.93,3.29)   ;
\draw   (149,81.5) .. controls (149,71.84) and (156.84,64) .. (166.5,64) .. controls (176.16,64) and (184,71.84) .. (184,81.5) .. controls (184,91.16) and (176.16,99) .. (166.5,99) .. controls (156.84,99) and (149,91.16) .. (149,81.5) -- cycle ;
\draw    (197.5,81) -- (232.5,81) ;
\draw [shift={(234.5,81)}, rotate = 180] [color={rgb, 255:red, 0; green, 0; blue, 0 }  ][line width=0.75]    (10.93,-3.29) .. controls (6.95,-1.4) and (3.31,-0.3) .. (0,0) .. controls (3.31,0.3) and (6.95,1.4) .. (10.93,3.29)   ;
\draw   (246,80.5) .. controls (246,70.84) and (253.84,63) .. (263.5,63) .. controls (273.16,63) and (281,70.84) .. (281,80.5) .. controls (281,90.16) and (273.16,98) .. (263.5,98) .. controls (253.84,98) and (246,90.16) .. (246,80.5) -- cycle ;
\draw    (294.5,80) -- (329.5,80) ;
\draw [shift={(331.5,80)}, rotate = 180] [color={rgb, 255:red, 0; green, 0; blue, 0 }  ][line width=0.75]    (10.93,-3.29) .. controls (6.95,-1.4) and (3.31,-0.3) .. (0,0) .. controls (3.31,0.3) and (6.95,1.4) .. (10.93,3.29)   ;
\draw   (341,81.5) .. controls (341,71.84) and (348.84,64) .. (358.5,64) .. controls (368.16,64) and (376,71.84) .. (376,81.5) .. controls (376,91.16) and (368.16,99) .. (358.5,99) .. controls (348.84,99) and (341,91.16) .. (341,81.5) -- cycle ;
\draw    (389.5,81) -- (424.5,81) ;
\draw [shift={(426.5,81)}, rotate = 180] [color={rgb, 255:red, 0; green, 0; blue, 0 }  ][line width=0.75]    (10.93,-3.29) .. controls (6.95,-1.4) and (3.31,-0.3) .. (0,0) .. controls (3.31,0.3) and (6.95,1.4) .. (10.93,3.29)   ;
\draw   (436,82.5) .. controls (436,72.84) and (444.02,65) .. (453.91,65) .. controls (463.8,65) and (471.82,72.84) .. (471.82,82.5) .. controls (471.82,92.16) and (463.8,100) .. (453.91,100) .. controls (444.02,100) and (436,92.16) .. (436,82.5) -- cycle ;

\draw (65,73) node [anchor=north west][inner sep=0.75pt]   [align=left] {0};
\draw (161,74) node [anchor=north west][inner sep=0.75pt]   [align=left] {7};
\draw (259,73) node [anchor=north west][inner sep=0.75pt]   [align=left] {2};
\draw (352,74) node [anchor=north west][inner sep=0.75pt]   [align=left] {8};
\draw (444.41,74) node [anchor=north west][inner sep=0.75pt]   [align=left] {10};
\draw (108,97) node [anchor=north west][inner sep=0.75pt]  [font=\small] [align=left] {$\displaystyle \mathbf{\frac{1}{1}}$};
\draw (202,98) node [anchor=north west][inner sep=0.75pt]  [font=\small] [align=left] {$\displaystyle \frac{\mathbf{1}\textcolor[rgb]{0.5,0.5,0.5}{1}}{\mathbf{0}\textcolor[rgb]{0.5,0.5,0.5}{1}}$};
\draw (292,94) node [anchor=north west][inner sep=0.75pt]  [font=\small] [align=left] {$\displaystyle \frac{\mathbf{0}\textcolor[rgb]{0.5,0.5,0.5}{11}}{\mathbf{1}\textcolor[rgb]{0.5,0.5,0.5}{01}}$};
\draw (387,96) node [anchor=north west][inner sep=0.75pt]  [font=\small] [align=left] {$\displaystyle \frac{\mathbf{0}\textcolor[rgb]{0.5,0.5,0.5}{011}}{\mathbf{1}\textcolor[rgb]{0.5,0.5,0.5}{101}}$};

\end{tikzpicture}
\caption{Directed graph of the steps of the algorithm for $m=22$.}\label{22digraph}
\end{figure}
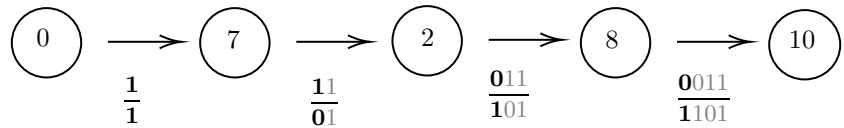
\end{center}
\begin{itemize}
    \item Step $0$: As always, we start with a carry value of 0, which is congruent to $0\bmod 3$, and we choose to add $m=22$ to the carry value, record $1\cdot3^0$ into $q$ and into $p$, subtract 1 from the carry value accordingly, and move to Step $1$.  The remaining unrecorded value is $(22-1)\cdot3^0=21\cdot3^0=7\cdot3^1$, which yields a carry value of $7$.  Note that in Figure \ref{22digraph}, Step 0 corresponds to the edge from the vertex 0 to the vertex 7, which is the new carry value, and the label on the edge denotes the $3^0$ digits of the base 3 expansions of $p$ (in the numerator) and $q$ (in the denominator).
    \item Step $1$: The carry value is $7\equiv1\pmod3$.  We record $0\cdot3^1$ into $q$ and $1\cdot3^1$ into $p$, and we subtract 1 from the carry value accordingly.  The remaining unrecorded value is now $(7-1)\cdot3^1=6\cdot3^1=2\cdot3^2$, which yields a carry value of $2$.  In Figure \ref{22digraph}, Step 1 corresponds to the edge from the vertex 7 to the vertex 2, the new carry value, and the new recorded digits in the edge label, $p$ in the numerator and $q$ in the denominator, are the $3^1$ digits.
    \item Step $2$: The carry value is $2\equiv2\bmod3$.  We add $m=22$ to the carry value, record $1\cdot3^2$ into $q$, and record $0\cdot3^2$ into $p$.  The remaining unrecorded value at this step is $(2+22)\cdot3^2=24\cdot3^2=8\cdot3^3$, which yields a carry value of $8$.  Step 2 corresponds to the edge from the vertex 2 to the vertex 8 with the $3^2$ digits of $p$ and $q$ recorded.
    \item Step $3$: The carry value is $8\equiv 2\pmod 3$, so we again add $m=22$ to the carry value, record $1\cdot3^3$ into $q$, and record $0\cdot3^3$ into $p$. The remaining unrecorded value is $(8+22)\cdot3^3=30\cdot3^3=10\cdot3^4$, which yields a carry value of $10$.  Step 3 corresponds to the edge from the vertex 8 to the vertex 10 with the $3^3$ digits of $p$ and $q$ recorded.
    \item Step $4$: The carry value is $10=3^2+3^0$, which is a sum of distinct powers of 3, so the algorithm terminates.
\end{itemize}
At this point, we append the base 3 representation of $10$, $[101]_3$, to the left of the digits for $p$, and we arrive at the final quotient $22=\frac{[1010011]_3}{[1101]_3}$.
\end{example}

In the example $m=22$, for vertices congruent to $0\bmod 3$, we make the first choice in the description of the algorithm (adding $m$) only in the first step, as the second option would loop back to the vertex 0 and keep us from moving toward our goal of obtaining a representation of $m$. If we continued beyond the carry value of 10 to draw edges for all possible paths originating at the vertex 0, we would have the extended graph in Figure \ref{subgraphoftransducer}.  Such an extended graph shows all representations of an integer $m$.  The particular representation of $m=22$ constructed above by stopping after the carry value of 10 is the same representation we obtain from following the path $0\rightarrow7\rightarrow2\rightarrow8\rightarrow10\rightarrow3\rightarrow1\rightarrow0$ in Figure \ref{subgraphoftransducer}, since the subpath $10\rightarrow3\rightarrow1\rightarrow0$ appends the same digits to $p$ as the digits in the base 3 representation of 10. We see this by observing that the subpath $10\rightarrow3\rightarrow1\rightarrow0$ contributes no digits to $q$ while contributing to $p$ the digits 101.

\begin{figure}[h!]
\tikzset{every picture/.style={line width=0.75pt}} 
\begin{tikzpicture}[x=0.75pt,y=0.75pt,yscale=-1,xscale=1]

\draw   (54,81.5) .. controls (54,71.84) and (61.84,64) .. (71.5,64) .. controls (81.16,64) and (89,71.84) .. (89,81.5) .. controls (89,91.16) and (81.16,99) .. (71.5,99) .. controls (61.84,99) and (54,91.16) .. (54,81.5) -- cycle ;
\draw    (102.5,81) -- (137.5,81) ;
\draw [shift={(139.5,81)}, rotate = 180] [color={rgb, 255:red, 0; green, 0; blue, 0 }  ][line width=0.75]    (10.93,-3.29) .. controls (6.95,-1.4) and (3.31,-0.3) .. (0,0) .. controls (3.31,0.3) and (6.95,1.4) .. (10.93,3.29)   ;
\draw   (149,81.5) .. controls (149,71.84) and (156.84,64) .. (166.5,64) .. controls (176.16,64) and (184,71.84) .. (184,81.5) .. controls (184,91.16) and (176.16,99) .. (166.5,99) .. controls (156.84,99) and (149,91.16) .. (149,81.5) -- cycle ;
\draw    (197.5,81) -- (232.5,81) ;
\draw [shift={(234.5,81)}, rotate = 180] [color={rgb, 255:red, 0; green, 0; blue, 0 }  ][line width=0.75]    (10.93,-3.29) .. controls (6.95,-1.4) and (3.31,-0.3) .. (0,0) .. controls (3.31,0.3) and (6.95,1.4) .. (10.93,3.29)   ;
\draw   (246,80.5) .. controls (246,70.84) and (253.84,63) .. (263.5,63) .. controls (273.16,63) and (281,70.84) .. (281,80.5) .. controls (281,90.16) and (273.16,98) .. (263.5,98) .. controls (253.84,98) and (246,90.16) .. (246,80.5) -- cycle ;
\draw    (294.5,80) -- (329.5,80) ;
\draw [shift={(331.5,80)}, rotate = 180] [color={rgb, 255:red, 0; green, 0; blue, 0 }  ][line width=0.75]    (10.93,-3.29) .. controls (6.95,-1.4) and (3.31,-0.3) .. (0,0) .. controls (3.31,0.3) and (6.95,1.4) .. (10.93,3.29)   ;
\draw   (341,81.5) .. controls (341,71.84) and (348.84,64) .. (358.5,64) .. controls (368.16,64) and (376,71.84) .. (376,81.5) .. controls (376,91.16) and (368.16,99) .. (358.5,99) .. controls (348.84,99) and (341,91.16) .. (341,81.5) -- cycle ;
\draw    (389.5,81) -- (424.5,81) ;
\draw [shift={(426.5,81)}, rotate = 180] [color={rgb, 255:red, 0; green, 0; blue, 0 }  ][line width=0.75]    (10.93,-3.29) .. controls (6.95,-1.4) and (3.31,-0.3) .. (0,0) .. controls (3.31,0.3) and (6.95,1.4) .. (10.93,3.29)   ;
\draw   (436,82.5) .. controls (436,72.84) and (444.02,65) .. (453.91,65) .. controls (463.8,65) and (471.82,72.84) .. (471.82,82.5) .. controls (471.82,92.16) and (463.8,100) .. (453.91,100) .. controls (444.02,100) and (436,92.16) .. (436,82.5) -- cycle ;
\draw    (486.5,81) -- (521.5,81) ;
\draw [shift={(523.5,81)}, rotate = 180] [color={rgb, 255:red, 0; green, 0; blue, 0 }  ][line width=0.75]    (10.93,-3.29) .. controls (6.95,-1.4) and (3.31,-0.3) .. (0,0) .. controls (3.31,0.3) and (6.95,1.4) .. (10.93,3.29)   ;
\draw   (533,82.5) .. controls (533,72.84) and (541.02,65) .. (550.91,65) .. controls (560.8,65) and (568.82,72.84) .. (568.82,82.5) .. controls (568.82,92.16) and (560.8,100) .. (550.91,100) .. controls (541.02,100) and (533,92.16) .. (533,82.5) -- cycle ;
\draw   (630,82.5) .. controls (630,72.84) and (638.02,65) .. (647.91,65) .. controls (657.8,65) and (665.82,72.84) .. (665.82,82.5) .. controls (665.82,92.16) and (657.8,100) .. (647.91,100) .. controls (638.02,100) and (630,92.16) .. (630,82.5) -- cycle ;
\draw    (583.5,83) -- (618.5,83) ;
\draw [shift={(620.5,83)}, rotate = 180] [color={rgb, 255:red, 0; green, 0; blue, 0 }  ][line width=0.75]    (10.93,-3.29) .. controls (6.95,-1.4) and (3.31,-0.3) .. (0,0) .. controls (3.31,0.3) and (6.95,1.4) .. (10.93,3.29)   ;
\draw    (550.5,56) .. controls (497.04,12.44) and (391.63,20.83) .. (364.3,55.93) ;
\draw [shift={(363.5,57)}, rotate = 305.84] [color={rgb, 255:red, 0; green, 0; blue, 0 }  ][line width=0.75]    (10.93,-3.29) .. controls (6.95,-1.4) and (3.31,-0.3) .. (0,0) .. controls (3.31,0.3) and (6.95,1.4) .. (10.93,3.29)   ;
\draw    (646.5,106) .. controls (603.71,178.64) and (146.11,129.5) .. (76.52,108.32) ;
\draw [shift={(75.5,108)}, rotate = 17.65] [color={rgb, 255:red, 0; green, 0; blue, 0 }  ][line width=0.75]    (10.93,-3.29) .. controls (6.95,-1.4) and (3.31,-0.3) .. (0,0) .. controls (3.31,0.3) and (6.95,1.4) .. (10.93,3.29)   ;
\draw    (77.5,56) .. controls (83.41,22.51) and (63.12,22.01) .. (63.47,55.45) ;
\draw [shift={(63.5,57)}, rotate = 268.36] [color={rgb, 255:red, 0; green, 0; blue, 0 }  ][line width=0.75]    (10.93,-3.29) .. controls (6.95,-1.4) and (3.31,-0.3) .. (0,0) .. controls (3.31,0.3) and (6.95,1.4) .. (10.93,3.29)   ;

\draw (65,73) node [anchor=north west][inner sep=0.75pt]   [align=left] {0};
\draw (161,74) node [anchor=north west][inner sep=0.75pt]   [align=left] {7};
\draw (259,73) node [anchor=north west][inner sep=0.75pt]   [align=left] {2};
\draw (352,74) node [anchor=north west][inner sep=0.75pt]   [align=left] {8};
\draw (444.41,74) node [anchor=north west][inner sep=0.75pt]   [align=left] {10};
\draw (546.41,74) node [anchor=north west][inner sep=0.75pt]   [align=left] {3};
\draw (641.41,73) node [anchor=north west][inner sep=0.75pt]   [align=left] {1};

\end{tikzpicture}
\caption{The subgraph of the multiplication transducer for $m=22$ in base 3 that encodes all representations of $22$ as a quotient of sums of distinct powers of 3.}\label{subgraphoftransducer}
\end{figure}
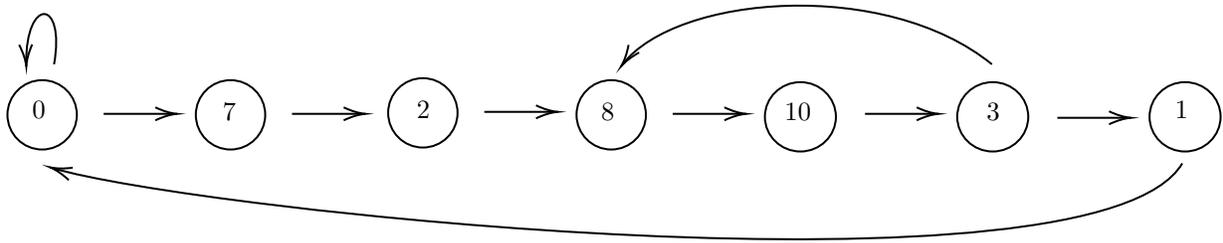

In general, the algorithm serves two main purposes for this paper.  First, if the algorithm terminates for a given integer $m\equiv1\pmod{3}$, as in Step 4 of Example \ref{22 example}, then the algorithm proves that $m$ is representable.  Second, if all possible paths back to $0$ are included in the graph produced by the algorithm for a given $m$, then the algorithm provides an explicit construction of all representations of $m$.  The directed graph involves two branches corresponding to the two choices in the algorithm.   Next we will discuss the additional information about an integer $m$ and its representations that can be gleaned from an extended graph.

The graph in Figure \ref{22digraph} is a subgraph of  the graph in Figure \ref{subgraphoftransducer}, while the graph in Figure \ref{subgraphoftransducer} can be generated as a subgraph of a multiplication transducer. A multiplication transducer, as described in \cite{BDT}, is a finite state automaton which performs the operation of multiplication by $m$ in base $b$ for any given positive integers $m$ and $b$. The multiplication transducer that performs multiplication by $m$ in base 3 can multiply $m\cdot q = p$ for any positive integer $q$ as follows: by reading the digits of the base 3 representation of $q$ from right to left, following the directed edges of the graph, and recording the digits of $p$ after each step.  This process continues until all digits have been read and the value carried over from the previous step is 0. As noted, the graph in Figure \ref{subgraphoftransducer} is a subgraph of the multiplication transducer that multiplies by $m=22$ in base 3, and this subgraph can be generated from the transducer as follows: by restricting edges to those that generate $q, p \in \mathcal{A}$ and subsequently restricting vertices to the carry values that are possible steps in the computation $m\cdot q=22\cdot q = p$ for $q, p \in \mathcal{A}$.

When exploring the representations of an integer $m$ through an expanded graph that includes the path back to the vertex 0, such as Figure \ref{subgraphoftransducer},  we can generate all possible representations of $m$ as a quotient of sums of distinct powers of 3. Each representation corresponds to a specific closed walk through the vertices of the digraph, starting and ending at the vertex 0. We note that a closed walk must include at least one step away from the vertex 0 in order to correspond to a valid representation, and by restricting our attention to elements of $\mathcal{A}$, we ensure that the initial step away from the vertex 0 will never be from the vertex 0 to itself. In the example above, we outlined the representation $22 =\frac{[1010011]_3}{[1101]_3} =\frac{814}{37}$, but taking a different closed walk, stepping up from the vertex 3 to the vertex 8 before returning to the vertex 0 in Figure \ref{subgraphoftransducer}, would give the representation $22 = \frac{[1010110011]_3}{[1101101]_3} =\frac{22198}{1009}$. 

\begin{definition}\label{decomposable and indecomposable}
An \textit{indecomposable walk} is a closed walk through the graph that contains the vertex 0 as the initial point and terminal point but not as an interior point, and an \textit{indecomposable representation} is a representation corresponding to such a walk.  A \textit{decomposable walk} is a closed walk through the graph that contains the vertex 0 not only as the initial point and terminal point but also as an interior point, and a \textit{decomposable representation} is a representation corresponding to a decomposable walk.  
\end{definition}

In the example of $m=22$, the representations given in the paragraph above Definition \ref{decomposable and indecomposable} are indecomposable representations, while $22 = \frac{[1010011\underline{1010011}]_3}{[1101\underline{0001101}]_3}$ is an example of a decomposable representation.  In general, a decomposable representation can be thought of as the concatenation of multiple indecomposable representations written in base 3. Before concatenating, we append copies of the digit 0 to the denominator of each indecomposable representation to make the lengths of the numerator and denominator match. To create the decomposable representation we gave for $m = 22$, we concatenated two copies of $\frac{[1010011]_3}{[1101]_3}$, appending three digits of 0 as placeholders in the denominator of the first (underlined) copy. This can also be thought of in base 10 as creating the representation $22 = \frac{3^7 \cdot 814 + 814}{3^7\cdot 37+ 37}$. In other words, a decomposable representation can be written as $\frac{3^kp_1+p_2}{3^kq_1 + q_2}$ for two representations $m = \frac{p_1}{q_1}$ and $m = \frac{p_2}{q_2}$, where $k$ is at least the length of the base 3 representation of $p_2$.


\section{Integers with only universal representations}\label{special families}

In our exploration of representations of integers as quotients of sums of distinct powers of 3, we have encountered many integers $m\equiv1\pmod{3}$ which have only universal representations.  We devote this section to such integers.  We begin in Section \ref{NewRepsSection} by describing how to build new universal representations from a given universal representation. In Section \ref{SpecialFamiliesSection}, we prove that four infinite families of integers have only universal representations.  Finally, we investigate the representations of integers of the form $m=3^n-2$ in Section \ref{3n-2Section} and of $m=100$ in Section \ref{RepsOf100Section}.

\subsection{Tools for Building New Representations}\label{NewRepsSection}

Given a representation of an integer $m$ as a quotient of sums of distinct powers of 3, there are several ways we can generate new representations.  Some aspects of this study can be handled quickly, using only properties of the set $\mathcal P$.  Before proceeding, it will be helpful to adopt an alternative description of Newman polynomials and to establish two lemmas. Write
\[
p(x) = 1 + \sum_{k=1}^{r-1} \epsilon_k x^k + x^r = \sum_{j=0}^m x^{a_j},
\]
where $0 = a_0 < a_1 < \cdots < a_m = r$, and let $\Delta(p): = \{a_i - a_j: 0 \le j < i \le m\}$.  The following lemma identifies exactly when a product of Newman polynomials is itself a Newman polynomial.

\begin{lemma}\label{L:51}
Suppose $s(x), t(x) \in \mathcal P$. Then $(st)(x) \in \mathcal P$  if and only if $\Delta(s) \cap \Delta(t) = \emptyset$.
\end{lemma}
\begin{proof}
If $s(x) = \sum\limits_{k=0}^m x^{a_k}$ and $t(x) = \sum\limits_{\ell=0}^n x^{b_\ell}$, then $(st)(x) =  \sum\limits_{k=0}^m \sum\limits_{\ell=0}^n x^{a_k+b_\ell}$.
The coefficient of $x^j$ in $(st)(x)$ is the number of times that $a_k + b_\ell =
j$, which is always non-negative and is in $\{0,1\}$ if and only if $a_k + b_{\ell} = a_{k'} + b_{\ell'}$ implies $k = k'$ and $\ell = \ell'$.  This happens if and only if
$a_k -  a_{k'} = b_{\ell'} -  b_{\ell}$ is equivalent to  $k = k'$ and $\ell = \ell' $, which happens if and only if  $\Delta(s) \cap \Delta(t) = \emptyset$.
\end{proof}

We remark here that one simple way to avoid common exponent gaps between two Newman polynomials $s(x),t(x)$ is to require that all exponent gaps in $s(x)$ are greater than the degree of $t(x)$, or vice versa.  This is a recurring idea in this paper and will be referenced several times in the remainder of this subsection as well as in the proof of Theorem \ref{EveryRepOf100Universal}.

We now present a different version of Lemma \ref{L:51} which can be used to view exponent gaps within Newman polynomials more generally from the perspective of universal representations.  This lemma can be proved by a direct computation, which is omitted.
\begin{lemma}\label{ForTheorem2.5}
Suppose $p_j(x), q_j(x) \in \mathcal P$ for all $1 \le j \le w$ and 
$p(x) =  \sum_j x^{n_j}p_j(x), q(x) =  \sum_j x^{n_j}q_j(x) \in \mathcal P$, with
no other conditions on the non-negative integers $n_j$.
Suppose further that $\frac{p_j(x)}{q_j(x)} = t(x)$ for all
$j$. Then $\frac{p(x)}{q(x)} = t(x)$ as well. Similarly, if 
$\frac{p_j(3)}{q_j(3)} = t(3)=  m$ for all $j$, then $\frac{p(3)}{q(3)} = m$.
\end{lemma}

Note that if we impose the additional condition that $n_j>\deg\left(\sum_{i<j}x^{n_i}p_i(x)\right)$ for each $j$, then we obtain a special case of Lemma \ref{ForTheorem2.5} that will be applied in the proof of Theorem \ref{EveryRepOf100Universal} to prove the universality of some representations of $100$.

Equipped with these lemmas, we now discuss two ways to build new representations from a given representation.  First, we can generate additional representations of the same integer $m$.  Every integer having one representation in fact has
infinitely many representations. Suppose $m =\frac {p(3)}{q(3)}$, where $p(x), q(x)\in\mathcal{P}$ and $\deg p(x), \deg q(x) \le T$. If $f(x)$ is any polynomial in $\mathcal{A}$ for which the gaps in exponents 
are greater than $T$, then the exponent gaps in $f(x)$ ensure that there are no possible cross-terms in the products $(fp)(x)$ and $(fq)(x)$.  By Lemma \ref{L:51}, we have that $(fp)(x), (fq)(x) \in \mathcal{A}$, and so $m = \frac {(fp)(3)}{(fq)(3)}$, giving a new representation of the same integer $m$, and this new representation is decomposable.

Secondly, we can work with a given representation of $m$ to obtain representations of other integers.  Suppose $m = \frac{p(3)}{q(3)}\in\mathcal{A}/\mathcal{A}$ and $f(x) \in \mathcal{P}$ is any polynomial for which $\Delta(f) \cap \Delta(p) = \emptyset$. Then $(fp)(x) \in \mathcal{P}$ by Lemma \ref{L:51}, so, multiplying $m$ by $f(x)$ and evaluating at $x=3$, we have that $f(3)\cdot m = \frac {(fp)(3)}{q(3)}$ is in $\mathcal{A}/\mathcal{A}$ and is a representation not of $m$ but of $f(3)\cdot m$.  A particular type of polynomial $f(x)$ that will yield such a representation of a new integer is a polynomial $f(x)$ with exponent gaps larger than the degree of $p(x)$.  More explicitly, if $\deg p(x) = r$ and $f(x)\in\mathcal{P}$ such that $\min\Delta(f)>r$, then $(fp)(x) \in \mathcal P$ and thus $f(3)\cdot m\in\mathcal{A}/\mathcal{A}$.

\subsection{Special families}\label{SpecialFamiliesSection}  

This section is dedicated to proving that several infinite families of integers $m\equiv1\pmod{3}$ have only universal representations and, moreover, that those universal representations must have a specific form.  There are certain values of $m$ for which the representation $m = \frac{p(3)}{q(3)}\in\mathcal{A}/\mathcal{A}$ gives very strong conditions on $p(x)$ and $q(x)$, and 
the following lemma will be useful as we prove these conditions.

\begin{lemma}\label{L:uni1}
Let $\mathcal{C},\mathcal{D}$ be sets of integers, and suppose one of the following two cases holds:
\begin{itemize}
    \item Case 1: $\mathcal C = \mathcal D = \{n,n+1,n+2\}$, where $n\in\mathbb{Z}$;
    \item Case 2: $\mathcal{C} = \{0,1\}\ \text{and} \ \mathcal D = \{-1,0,1,2\}$.
\end{itemize}
Let $a(x) = \sum_i a_i x^i$ and $b(x) = \sum_i b_i x^i$, where $a_i \in \mathcal C$ and $b_i \in \mathcal D$.  If $a(3) = b(3)$, then $a_i = b_i$ for all $i$, so $a(x) = b(x)$.
\end{lemma}
\begin{proof}
We prove the result by induction on $\operatorname{max}\{\deg a(x), \deg b(x)\}$; the assertion is clear when this
maximum is 0. Since $a(3)=b(3)$, we know $\sum_i a_i 3^i = \sum_i b_i 3^i$, and this gives $a_0 \equiv b_0 \pmod 3$.
In the first case, this means that $a_0 = b_0$; we may write
$a(x) = a_0 + x\tilde a(x)$ and $b(x) = b_0 + x\tilde b(x)$ and apply the inductive
argument to $\tilde a, \tilde b$. In the second case, we first note that $b_0 \notin \{-1,2\}$ and, 
again, $a_0=b_0$ and repeat the argument.
\end{proof}

These results allow us to make explicit statements about representations of certain special families of integers congruent to $1\bmod 3$. If $m$ belongs to one of the families described below in Theorem \ref{T:onlyuniversal}, then $m$ has only universal representations, and they must be of a particular kind.

\begin{theorem}\label{T:onlyuniversal} Let $p(x),q(x)\in\mathcal{P}$.
\begin{enumerate}
\item\label{all easy part one} If $3^r+1 = \frac{p(3)}{q(3)}$, then $p(x) = (x^r+1)q(x)$.
\item\label{all easy part two} If $3^{rk} + 3^{(r-1)k} + \dots + 3^k +1 = \frac{p(3)}{q(3)}$, then 
\[
p(x) = (x^{rk}+x^{(r-1)k} + \dots + x^k + 1)q(x) = \frac{x^{(r+1)k} - 1}{x^k -1} q(x).
\]
\item\label{all easy part three} If $3^{2rk} - 3^{(2r-1)k} + \dots - 3^k +1 = \frac{p(3)}{q(3)}$, then 
\[
p(x) = (x^{2rk}-x^{(2r-1)k} + \dots - x^k + 1)q(x)  = \frac{x^{(2r+1)k} + 1}{x^k +1} q(x).
\]
\item\label{all easy part four} If $r > s > 0$ and  $3^r - 3^s + 1 = \frac{p(3)}{q(3)}$, 
then 
\[
p(x) = (x^r-x^s+1)q(x).
\]
\end{enumerate}
In each of these four cases, there exist $p(x), q(x) \in \mathcal P$ which satisfy the particular condition. 
\end{theorem}

\begin{proof}
The proofs of the various parts of this theorem are similar and begin by letting $p(x),q(x)\in\mathcal{P}.$

For Part \ref{all easy part one}, assume $3^r+1=\frac{p(3)}{q(3)}$, and consider the polynomials $p(x)$ and $(x^r+1)q(x)$. Observe that the coefficients of $p(x)$ are in $\{0,1\}$ and the coefficients of $(x^r+1)q(x)$
are in $\{0,1,2\}$. Since $p(3) = (3^r+1)q(3)$,  we know from Lemma \ref{L:uni1} that $p(x) = (x^r+1)q(x)$. A simple example with $p(x),q(x)\in\mathcal{P}$ comes from taking $q(x)=1$.

Similarly, in Part \ref{all easy part two}, assume $3^{rk} + 3^{(r-1)k} + \dots + 3^k +1 = \frac{p(3)}{q(3)}$, and consider the polynomials $(x^k-1)p(x)$ and $(x^{(r+1)k}-1)q(x)$. Since $p(x), q(x) \in \mathcal P$, it follows that both $(x^k-1)p(x)$ and $(x^{(r+1)k}-1)q(x)$ have coefficients in $\{-1,0,1\}$. We take their quotient at $x=3$ and use the hypothesis regarding $\frac{p(3)}{q(3)}$ to obtain
\[
\frac{(3^{(r+1)k}-1)q(3)}{(3^k-1)p(3)} = \frac{(3^{(r+1)k}-1)}{(3^k-1)} \left(\frac{p(3)}{q(3)}\right)^{-1} = 
\frac{3^{(r+1)k} - 1}{3^{(r+1)k} - 1} = 1.
\]
Thus $(3^k-1)p(3)=(3^{(r+1)k}-1)q(3)$, and it follows from Lemma \ref{L:uni1} that $(x^k-1)p(x)=(x^{(r+1)k}-1)q(x)$.  Again, a simple example with $p(x),q(x)\in\mathcal{P}$ comes from taking $q(x)=1$.

In Part \ref{all easy part three}, assume $3^{2rk} - 3^{(2r-1)k} + \dots - 3^k +1 = \frac{p(3)}{q(3)}$, and consider the polynomials $(x^k+1)p(x)$ and $(x^{(2r+1)k}+1)q(x)$. The coefficients of these products are in $\{0,1,2\}$. As in the proof of Part \ref{all easy part two}, we take the quotient of these two polynomials at $x=3$ and use the hypothesis regarding $\frac{p(3)}{q(3)}$ to obtain  
\[
(3^k+1)p(3)=(3^{(2r+1)k}+1)q(3).
\]
Again by Lemma \ref{L:uni1}, it follows that $(x^k+1)p(x)=(x^{(2r+1)k}+1)q(x)$.  As an example, take $q(x) = x^k+1$. Then $p(x) =  x^{(2r+1)k}+1$, and we can see that $p(x), q(x) \in \mathcal{P}$, and $\frac{p(x)}{q(x)} = x^{2rk} - x^{(2r-1)k} + \dots - x^k +1$.
 
Lastly, in Part 4, assume $3^r - 3^s + 1 = \frac{p(3)}{q(3)}$, and consider the polynomials $p(x)$ and $(x^r-x^s+1)q(x)$. By hypothesis
$p(x),q(x)\in \mathcal P$, so we can write $q(x)=\sum\limits_{j=0}^v  x^{b_j}$, and 
\[
(x^r-x^s+1)q(x) = (x^r+1)\left(\sum_{j=0}^v  x^{b_j}\right) - x^s\left(\sum_{j=0}^v  x^{b_j}\right).
\]
We see that $(x^r-x^s+1)q(x)$ is the difference of a polynomial with coefficients in $\{0,1,2\}$ and a polynomial with coefficients in $\{0,1\}$, so the coefficients of $(x^r-x^s+1)q(x)$ are in $\{-1,0,1,2\}$. Since $p(3) = (3^r-3^s+1)q(3)$, we can use Lemma \ref{L:uni1} one final time to conclude that $p(x)=(x^r-x^s+1)q(x)$.  Finally, an example comes from taking $q(x)=x^{r-1}+\cdots+x+1$, and then $p(x)=x^{2r-1}+\cdots +x^{r+s}+x^{s-1}+\cdots+x+1$.
\end{proof}

Among the integers congruent to $1 \bmod 3$ and less than or equal to $121$, the following
are covered by one or more of the parts of Theorem \ref{T:onlyuniversal} and thus have only universal representations: 4, 7, 10, 13, 19, 25,
28, 40, 55, 61,  73, 79, 82, 91, 121. The integers in intervals excluded by Theorem \ref{T:feasible} are
16, 43, 46, 49, 52. This leaves open 22, 31, 34, 37, 58, 64, 67, 70,
76, 85, 88, 94, 97, 103, 106, 109, 112, 115, 118. The first open two have quite different behavior; as we shall see in Section \ref{OnlyLocalSection}, 22 has only local representations, but we already know from Example \ref{31 example} that 31 can go either way!


\subsection{Integers of the form $3^n-2$}\label{3n-2Section}
Here we describe the representations of integers of the form $m=3^n-2$, where $n\geq3$.  First note that any integer of this form is covered by Theorem \ref{T:onlyuniversal} Part \ref{all easy part four}, since $3^n-2=3^n-3+1$.  Hence we know that all representations of such an integer must be universal.  We now describe all possible representations for $m=3^n-2$ and give the specific polynomials $p(x)$ and $q(x)$ such that $m=\frac{p(3)}{q(3)}$.

Suppose that $m=3^n-2=\frac{p(3)}{q(3)}$, where $p(x)=1+\sum\limits_{i=1}^{w-1}c_ix^i+x^w$ and $q(x)=1+\sum\limits_{i=1}^{s-1}d_ix^i+x^s$ with $c_i,d_i\in\{0,1\}$.  Let $g(x):=\frac{p(x)}{q(x)}\in\mathbb{Z}[x]$.  Then by Theorem \ref{T:feasible}, we know
\[
\frac{2}{3}\cdot 3^{w-s}<3^n-2<\frac{3}{2}\cdot3^{w-s},
\]
so $n=w-s=\deg g(x)$.

Since the constant terms and the coefficients of $x^w$ in $p(x)$ and $x^s$ in $q(x)$ are 1, we may then write $g(x)=1+\sum\limits_{i=1}^{n-1}e_ix^i+x^n$.  Equating coefficients of powers of $x$ in $p(x)=g(x)q(x)$ now gives the system of equations
$$c_i=\begin{cases}
d_i+e_i &\mbox{if }i=1,\\
d_i+\sum\limits_{j=1}^{i-1} d_{i-j}e_j+e_i &\mbox{if }1<i<n,\\
d_i+\sum\limits_{j=1}^{i-1}d_{i-j}e_j+1 &\mbox{if }i=n,\\
d_i+\sum\limits_{j=1}^{n-1}d_{i-j}e_j+d_{i-n}  &\mbox{if }n< i \leq w.
\end{cases}$$
We now use this system of equations to determine the coefficients $c_i, d_i$, and $e_i$ in the polynomials $p(x), q(x)$, and $g(x)$. We will approach this in sections: \begin{itemize}
    \item We first consider the coefficients corresponding to $i=1$ and then $2\leq i\leq n-1$.
    \item Next, we consider the coefficients corresponding to $i=n$. 
    \item After dealing with $i = n$, we reach a point where we can choose either to continue toward the end of the algorithm or to carry the value $\frac{m-1}{2}$ repeatedly before continuing. We consider the choice of carrying $\frac{m-1}{2}$ repeatedly $k$ times, for $n+1\leq i \leq n+k$.
    \item We then consider the remaining coefficients, when $n+k < i$.
    \end{itemize}

We begin by considering $g(3)$. Since
\begin{align*}
3^n-2=g(3)&=1+3e_1+9e_2+27e_3+\cdots+3^{n-1}e_{n-1}+3^n,\\
-3&=3e_1+9e_2+27e_3+\cdots+3^{n-1}e_{n-1},\\
-1&=e_1+3e_2+9e_3+\cdots+3^{n-2}e_{n-1},
\end{align*}
so $e_1\equiv 2\pmod 3$.  Since $e_1=c_1-d_1$, we see $c_1=0, d_1=1$, and $e_1=-1$.  Next observe that $3e_2\equiv 0\pmod 9$, so $e_2\equiv0\pmod 3$.  The second equation in our system is $c_2=d_2-1+e_2$.  Since $c_2,d_2\in\{0,1\}$, we must have $c_2=0, d_2=1$, and $e_2=0$.

Suppose that $2\leq i\leq n-1$ and that for all $j$ with $2\leq j<i$, we have $c_j=0, d_j=1$, and $e_j=0$.  Consider the $i$-th equation
\[
c_i=d_i+d_{i-1}e_1+d_{i-2}e_2+\cdots+d_2e_{i-2}+d_1e_{i-1}+e_i.
\]
This is 
\[
c_i=d_i+1(-1)+1(0)+\cdots+1(0)+1(0)+e_i,
\]
or, more simply, $c_i=d_i-1+e_i$.  We know $3^{i-1}e_i\equiv0\pmod {3^i}$, so $e_i\equiv0\pmod 3$. Then $c_i=0, d_i=1$, and $e_i=0$.  Thus we have $c_1=0, d_1=1$, and $e_1=-1$; and $c_2=\cdots=c_{n-1}=0, d_2=\cdots=d_{n-1}=1$, and $e_2=\cdots=e_{n-1}=0$.

Now consider the $n$-th equation
\[
c_n=d_n+d_{n-1}e_1+d_{n-2}e_2+\cdots+d_2e_{n-2}+d_1e_{n-1}+1.
\]
This is 
\[
c_n=d_n+1(-1)+0+\cdots+0+0+1, 
\]
so $c_n=d_n$.  We have $g(x)=1+\sum\limits_{i=1}^{n-1}e_ix^i+x^n=1-x+x^n$.  

As we complete the description of representations of $3^n-2$, we will refer to the digraph in Figure \ref{subgraphof3n_2}. Choosing $c_n=d_n=0$ uniquely determines a universal representation for $3^n-2$, and it is the same representation that comes from choosing to carry $\frac{m-1}{6}$ at the first possible opportunity after carrying $\frac{m-1}{2}$ (without ever taking the loop at the vertex $\frac{m-1}{2}$ in Figure \ref{subgraphof3n_2}).  It is straightforward to prove by induction that $\frac{m-1}{6}$ is a sum of distinct powers of 3 when $m=3^n-2$, and therefore choosing to carry $\frac{m-1}{6}$ without taking the loop leads to the immediate termination of the algorithm.

\begin{center}
\begin{figure}[h!]
\tikzset{every picture/.style={line width=0.75pt}} 

\begin{tikzpicture}[x=0.75pt,y=0.75pt,yscale=-.9,xscale=.9]

\draw   (21,96.12) .. controls (21,83.79) and (31.67,73.79) .. (44.84,73.79) .. controls (58,73.79) and (68.68,83.79) .. (68.68,96.12) .. controls (68.68,108.45) and (58,118.45) .. (44.84,118.45) .. controls (31.67,118.45) and (21,108.45) .. (21,96.12) -- cycle ;
\draw    (75,96.5) -- (99,96.5) ;
\draw [shift={(101,96.5)}, rotate = 180] [color={rgb, 255:red, 0; green, 0; blue, 0 }  ][line width=0.75]    (10.93,-3.29) .. controls (6.95,-1.4) and (3.31,-0.3) .. (0,0) .. controls (3.31,0.3) and (6.95,1.4) .. (10.93,3.29)   ;
\draw   (236.77,91.12) .. controls (236.77,78.79) and (247.45,68.79) .. (260.61,68.79) .. controls (273.78,68.79) and (284.45,78.79) .. (284.45,91.12) .. controls (284.45,103.45) and (273.78,113.45) .. (260.61,113.45) .. controls (247.45,113.45) and (236.77,103.45) .. (236.77,91.12) -- cycle ;
\draw    (289,93.5) -- (312.25,94.4) ;
\draw [shift={(314.24,94.48)}, rotate = 182.23] [color={rgb, 255:red, 0; green, 0; blue, 0 }  ][line width=0.75]    (10.93,-3.29) .. controls (6.95,-1.4) and (3.31,-0.3) .. (0,0) .. controls (3.31,0.3) and (6.95,1.4) .. (10.93,3.29)   ;
\draw   (323.91,87.84) .. controls (323.91,75.51) and (334.58,65.51) .. (347.75,65.51) .. controls (360.92,65.51) and (371.59,75.51) .. (371.59,87.84) .. controls (371.59,100.18) and (360.92,110.18) .. (347.75,110.18) .. controls (334.58,110.18) and (323.91,100.18) .. (323.91,87.84) -- cycle ;
\draw   (638.32,84.12) .. controls (638.32,71.79) and (649,61.79) .. (662.16,61.79) .. controls (675.33,61.79) and (686,71.79) .. (686,84.12) .. controls (686,96.45) and (675.33,106.45) .. (662.16,106.45) .. controls (649,106.45) and (638.32,96.45) .. (638.32,84.12) -- cycle ;
\draw    (359.33,54.03) .. controls (372.75,5) and (337.26,16.65) .. (342.05,54.82) ;
\draw [shift={(342.3,56.58)}, rotate = 261.19] [color={rgb, 255:red, 0; green, 0; blue, 0 }  ][line width=0.75]    (10.93,-3.29) .. controls (6.95,-1.4) and (3.31,-0.3) .. (0,0) .. controls (3.31,0.3) and (6.95,1.4) .. (10.93,3.29)   ;
\draw    (663,116.5) .. controls (554.57,178.72) and (145.67,156.96) .. (53.02,130.34) ;
\draw [shift={(51.65,129.94)}, rotate = 16.6] [color={rgb, 255:red, 0; green, 0; blue, 0 }  ][line width=0.75]    (10.93,-3.29) .. controls (6.95,-1.4) and (3.31,-0.3) .. (0,0) .. controls (3.31,0.3) and (6.95,1.4) .. (10.93,3.29)   ;
\draw    (54.37,63.58) .. controls (62.43,20.84) and (34.79,20.2) .. (35.26,62.87) ;
\draw [shift={(35.3,64.85)}, rotate = 268.25] [color={rgb, 255:red, 0; green, 0; blue, 0 }  ][line width=0.75]    (10.93,-3.29) .. controls (6.95,-1.4) and (3.31,-0.3) .. (0,0) .. controls (3.31,0.3) and (6.95,1.4) .. (10.93,3.29)   ;
\draw   (106,96.12) .. controls (106,83.79) and (116.67,73.79) .. (129.84,73.79) .. controls (143,73.79) and (153.68,83.79) .. (153.68,96.12) .. controls (153.68,108.45) and (143,118.45) .. (129.84,118.45) .. controls (116.67,118.45) and (106,108.45) .. (106,96.12) -- cycle ;
\draw    (163,96.5) -- (182,97.4) ;
\draw [shift={(184,97.5)}, rotate = 182.73] [color={rgb, 255:red, 0; green, 0; blue, 0 }  ][line width=0.75]    (10.93,-3.29) .. controls (6.95,-1.4) and (3.31,-0.3) .. (0,0) .. controls (3.31,0.3) and (6.95,1.4) .. (10.93,3.29)   ;
\draw    (207,96.5) -- (228,96.5) ;
\draw [shift={(230,96.5)}, rotate = 180] [color={rgb, 255:red, 0; green, 0; blue, 0 }  ][line width=0.75]    (10.93,-3.29) .. controls (6.95,-1.4) and (3.31,-0.3) .. (0,0) .. controls (3.31,0.3) and (6.95,1.4) .. (10.93,3.29)   ;
\draw    (377,90.5) -- (401,90.5) ;
\draw [shift={(403,90.5)}, rotate = 180] [color={rgb, 255:red, 0; green, 0; blue, 0 }  ][line width=0.75]    (10.93,-3.29) .. controls (6.95,-1.4) and (3.31,-0.3) .. (0,0) .. controls (3.31,0.3) and (6.95,1.4) .. (10.93,3.29)   ;
\draw   (542.77,86.12) .. controls (542.77,73.79) and (553.45,63.79) .. (566.61,63.79) .. controls (579.78,63.79) and (590.45,73.79) .. (590.45,86.12) .. controls (590.45,98.45) and (579.78,108.45) .. (566.61,108.45) .. controls (553.45,108.45) and (542.77,98.45) .. (542.77,86.12) -- cycle ;
\draw    (603,88.5) -- (626.25,89.4) ;
\draw [shift={(628.24,89.48)}, rotate = 182.23] [color={rgb, 255:red, 0; green, 0; blue, 0 }  ][line width=0.75]    (10.93,-3.29) .. controls (6.95,-1.4) and (3.31,-0.3) .. (0,0) .. controls (3.31,0.3) and (6.95,1.4) .. (10.93,3.29)   ;
\draw   (408,90.12) .. controls (408,77.79) and (418.67,67.79) .. (431.84,67.79) .. controls (445,67.79) and (455.68,77.79) .. (455.68,90.12) .. controls (455.68,102.45) and (445,112.45) .. (431.84,112.45) .. controls (418.67,112.45) and (408,102.45) .. (408,90.12) -- cycle ;
\draw    (465,90.5) -- (484,91.4) ;
\draw [shift={(486,91.5)}, rotate = 182.73] [color={rgb, 255:red, 0; green, 0; blue, 0 }  ][line width=0.75]    (10.93,-3.29) .. controls (6.95,-1.4) and (3.31,-0.3) .. (0,0) .. controls (3.31,0.3) and (6.95,1.4) .. (10.93,3.29)   ;
\draw    (513,91.5) -- (534,91.5) ;
\draw [shift={(536,91.5)}, rotate = 180] [color={rgb, 255:red, 0; green, 0; blue, 0 }  ][line width=0.75]    (10.93,-3.29) .. controls (6.95,-1.4) and (3.31,-0.3) .. (0,0) .. controls (3.31,0.3) and (6.95,1.4) .. (10.93,3.29)   ;

\draw (39.34,87.62) node [anchor=north west][inner sep=0.75pt]   [align=left] {0};
\draw (657.3,75.9) node [anchor=north west][inner sep=0.75pt]   [align=left] {1};
\draw (242.54,81.67) node [anchor=north west][inner sep=0.75pt]  [font=\tiny] [align=left] {$\displaystyle \frac{m-3}{2}$};
\draw (330.54,78.67) node [anchor=north west][inner sep=0.75pt]  [font=\tiny] [align=left] {$\displaystyle \frac{m-1}{2}$};
\draw (112.34,84.62) node [anchor=north west][inner sep=0.75pt]  [font=\tiny] [align=left] {$\displaystyle \frac{m-1}{3}$};
\draw (561.54,81.67) node [anchor=north west][inner sep=0.75pt]  [font=\small] [align=left] {4};
\draw (414.34,78.62) node [anchor=north west][inner sep=0.75pt]  [font=\tiny] [align=left] {$\displaystyle \frac{m-1}{6}$};
\draw (189,93.5) node [anchor=north west][inner sep=0.75pt]   [align=left] {...};
\draw (492,88.5) node [anchor=north west][inner sep=0.75pt]   [align=left] {...};

\end{tikzpicture}
\caption{The subgraph of the multiplication transducer for $m=3^n-2$ with $n\geq 3$ in base $3$ that encodes all representations of $m$ as a quotient of sums of distinct powers of 3.}\label{subgraphof3n_2}
\end{figure}
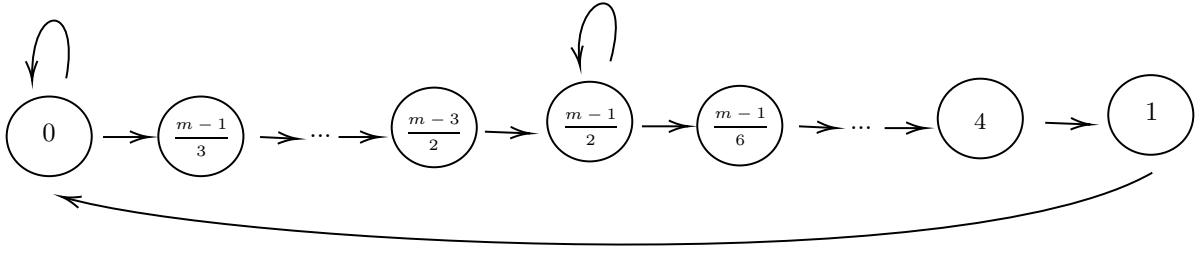
\end{center}

We now show that choosing to take the loop $k$ times for any positive integer $k$ will still result in a universal representation.
Suppose we choose $c_n=d_n=1$.  Then the $(n+1)$-th equation in our system is
\[
c_{n+1}=d_{n+1}+d_ne_1+d_{n-1}e_2+\cdots+d_2e_{n-1}+d_1,
\]
which is
\[
c_{n+1}=d_{n+1}+1(-1)+0\cdots+0+1, 
\]
so $c_{n+1}=d_{n+1}$.  Choose $c_{n+1}=d_{n+1}=1$.  The $(n+2)$-th equation is
\[
c_{n+2}=d_{n+2}+d_{n+1}e_1+d_ne_2+\cdots+d_3e_{n-1}+d_2,
\]
which is
\[
c_{n+2}=d_{n+2}+1(-1)+0+\cdots+0+1, 
\]
so $c_{n+2}=d_{n+2}$.  In general, suppose that $i\geq 1$ and for all $1\leq j<i$, we have chosen $c_{n+j}=d_{n+j}=1$.  Then the $(n+i)$-th equation is
\[
c_{n+i}=d_{n+i}+d_{n+i-1}e_1+d_{n+i-2}e_2+\cdots+d_{i+1}e_{n-1}+d_i,
\]
which is
\[
c_{n+i}=d_{n+i}+1(-1)+0+\cdots+0+1, 
\]
so $c_{n+i}=d_{n+i}$.  Suppose that for all $i$ with $1\leq i\leq k-1$, we choose $c_{n+i}=d_{n+i}=1$.  Then we have $c_1=0, d_1=1$, and $e_1=-1$; and $c_2=\cdots=c_{n-1}=0, c_n=\cdots=c_{n+k-1}=1, d_2=\cdots=d_{n-1}=1, d_n=\cdots=d_{n+k-1}=1$, and $e_2=\cdots=e_{n-1}=0$.  Then the $(n+k)$-th equation is
\[
c_{n+k}=d_{n+k}+d_{n+k-1}e_1+d_{n+k-2}e_2+\cdots+d_{k+1}e_{n-1}+d_k,
\]
which is 
\[
c_{n+k}=d_{n+k}+1(-1)+0+\cdots+0+1,
\]
so $c_{n+k}=d_{n+k}$.  We have now taken the loop from the vertex $\frac{m-1}{2}$ to itself $k$ times.  We choose $c_{n+k}=d_{n+k}=0$ to stop taking the loop and allow the algorithm to terminate.  Similar arguments to those outlined above show that $c_{n+k+1}=\cdots=c_{n+k+(n-1)}=1$ and $d_{n+k+1}=\cdots=d_{n+k+(n-1)}=0$.
Then we use the values we found for the coefficients $c_i$ and $d_i$ to write down the polynomials involved in the representation of $3^n-2$ that corresponds to taking the loop $k$ times.  These polynomials are $p_k(x)=1+x^n+\cdots+x^{n+k-1}+x^{n+k+1}+\cdots+x^{n+k+(n-1)}$ and $q_k(x)=1+x+\cdots+x^{n+k-1}$. We have $3^n-2=\frac{p_k(3)}{q_k(3)}$, and $x^n -2 = \frac{p_k(x)}{q_k(x)}$, giving a universal representation for $3^n-2$, and we note that we can obtain such a representation for any positive integer k.

\subsection{Representations of 100}\label{RepsOf100Section}
Now let us consider $m=100$, which is not covered by any of the theorems previously established in this section.  In fact, $m=100$ is the smallest integer that has only universal representations but is not covered by these theorems.  

Following the algorithm for $m=100$, we note that at the end of Step 1 we have a carry value of $33$.  As stated in Section \ref{sec:algorithm}, whenever the carry value is a multiple of 3, we have a choice of two directions.  Here, in Step $2$, we have a choice to move forward carrying either $44$ or $11$. We will refer to these choices in general as a \textit{step up} and a \textit{step down} from the current carry value, where the step up corresponds to the choice to add $m$ to the carry value and record $1\cdot3^i$ into $q$ and into $p$, while the step down corresponds to the choice to record $0\cdot3^i$ into $q$ and into $p$.   Returning to the algorithm for $100$, we see in Figure \ref{digraph for 100} that if we choose to step up from 33 and move forward carrying $44$, this will never lead us back to 0. Thus, in an effort to continue our search for representations of $100$, we step down from $33$ to $11$.

\begin{figure}[h!]
\tikzset{every picture/.style={line width=0.75pt}} 

\begin{tikzpicture}[x=0.75pt,y=0.75pt,yscale=-1,xscale=1]

\draw   (26,218.73) .. controls (26,209.92) and (32.98,202.78) .. (41.59,202.78) .. controls (50.2,202.78) and (57.18,209.92) .. (57.18,218.73) .. controls (57.18,227.55) and (50.2,234.69) .. (41.59,234.69) .. controls (32.98,234.69) and (26,227.55) .. (26,218.73) -- cycle ;
\draw    (69.2,218.28) -- (100.16,218.28) ;
\draw [shift={(102.16,218.28)}, rotate = 180] [color={rgb, 255:red, 0; green, 0; blue, 0 }  ][line width=0.75]    (10.93,-3.29) .. controls (6.95,-1.4) and (3.31,-0.3) .. (0,0) .. controls (3.31,0.3) and (6.95,1.4) .. (10.93,3.29)   ;
\draw   (110.62,218.73) .. controls (110.62,209.92) and (117.6,202.78) .. (126.21,202.78) .. controls (134.82,202.78) and (141.8,209.92) .. (141.8,218.73) .. controls (141.8,227.55) and (134.82,234.69) .. (126.21,234.69) .. controls (117.6,234.69) and (110.62,227.55) .. (110.62,218.73) -- cycle ;
\draw    (153.83,217.37) -- (184.79,217.37) ;
\draw [shift={(186.79,217.37)}, rotate = 180] [color={rgb, 255:red, 0; green, 0; blue, 0 }  ][line width=0.75]    (10.93,-3.29) .. controls (6.95,-1.4) and (3.31,-0.3) .. (0,0) .. controls (3.31,0.3) and (6.95,1.4) .. (10.93,3.29)   ;
\draw   (197.03,217.82) .. controls (197.03,209.01) and (204.01,201.87) .. (212.62,201.87) .. controls (221.23,201.87) and (228.21,209.01) .. (228.21,217.82) .. controls (228.21,226.64) and (221.23,233.78) .. (212.62,233.78) .. controls (204.01,233.78) and (197.03,226.64) .. (197.03,217.82) -- cycle ;
\draw    (240.23,217.37) -- (271.19,217.37) ;
\draw [shift={(273.19,217.37)}, rotate = 180] [color={rgb, 255:red, 0; green, 0; blue, 0 }  ][line width=0.75]    (10.93,-3.29) .. controls (6.95,-1.4) and (3.31,-0.3) .. (0,0) .. controls (3.31,0.3) and (6.95,1.4) .. (10.93,3.29)   ;
\draw   (281.66,218.73) .. controls (281.66,209.92) and (288.64,202.78) .. (297.24,202.78) .. controls (305.85,202.78) and (312.83,209.92) .. (312.83,218.73) .. controls (312.83,227.55) and (305.85,234.69) .. (297.24,234.69) .. controls (288.64,234.69) and (281.66,227.55) .. (281.66,218.73) -- cycle ;
\draw    (324.86,218.28) -- (355.82,218.28) ;
\draw [shift={(357.82,218.28)}, rotate = 180] [color={rgb, 255:red, 0; green, 0; blue, 0 }  ][line width=0.75]    (10.93,-3.29) .. controls (6.95,-1.4) and (3.31,-0.3) .. (0,0) .. controls (3.31,0.3) and (6.95,1.4) .. (10.93,3.29)   ;
\draw   (366.28,219.65) .. controls (366.28,210.83) and (373.42,203.69) .. (382.23,203.69) .. controls (391.04,203.69) and (398.19,210.83) .. (398.19,219.65) .. controls (398.19,228.46) and (391.04,235.6) .. (382.23,235.6) .. controls (373.42,235.6) and (366.28,228.46) .. (366.28,219.65) -- cycle ;
\draw    (411.27,218.28) -- (442.22,218.28) ;
\draw [shift={(444.22,218.28)}, rotate = 180] [color={rgb, 255:red, 0; green, 0; blue, 0 }  ][line width=0.75]    (10.93,-3.29) .. controls (6.95,-1.4) and (3.31,-0.3) .. (0,0) .. controls (3.31,0.3) and (6.95,1.4) .. (10.93,3.29)   ;
\draw   (452.69,219.65) .. controls (452.69,210.83) and (459.83,203.69) .. (468.64,203.69) .. controls (477.45,203.69) and (484.59,210.83) .. (484.59,219.65) .. controls (484.59,228.46) and (477.45,235.6) .. (468.64,235.6) .. controls (459.83,235.6) and (452.69,228.46) .. (452.69,219.65) -- cycle ;
\draw   (539.09,219.65) .. controls (539.09,210.83) and (546.24,203.69) .. (555.05,203.69) .. controls (563.86,203.69) and (571,210.83) .. (571,219.65) .. controls (571,228.46) and (563.86,235.6) .. (555.05,235.6) .. controls (546.24,235.6) and (539.09,228.46) .. (539.09,219.65) -- cycle ;
\draw    (497.67,220.1) -- (528.63,220.1) ;
\draw [shift={(530.63,220.1)}, rotate = 180] [color={rgb, 255:red, 0; green, 0; blue, 0 }  ][line width=0.75]    (10.93,-3.29) .. controls (6.95,-1.4) and (3.31,-0.3) .. (0,0) .. controls (3.31,0.3) and (6.95,1.4) .. (10.93,3.29)   ;
\draw    (553.79,241.07) .. controls (515.87,306.97) and (112.14,262.95) .. (47.02,243.48) ;
\draw [shift={(45.15,242.9)}, rotate = 18.04] [color={rgb, 255:red, 0; green, 0; blue, 0 }  ][line width=0.75]    (10.93,-3.29) .. controls (6.95,-1.4) and (3.31,-0.3) .. (0,0) .. controls (3.31,0.3) and (6.95,1.4) .. (10.93,3.29)   ;
\draw    (46.93,195.48) .. controls (52.17,165.1) and (34.31,164.49) .. (34.43,194.52) ;
\draw [shift={(34.46,196.39)}, rotate = 268.4] [color={rgb, 255:red, 0; green, 0; blue, 0 }  ][line width=0.75]    (10.93,-3.29) .. controls (6.95,-1.4) and (3.31,-0.3) .. (0,0) .. controls (3.31,0.3) and (6.95,1.4) .. (10.93,3.29)   ;
\draw    (141.8,202.78) -- (186.42,175.55) ;
\draw [shift={(188.12,174.51)}, rotate = 148.61] [color={rgb, 255:red, 0; green, 0; blue, 0 }  ][line width=0.75]    (10.93,-3.29) .. controls (6.95,-1.4) and (3.31,-0.3) .. (0,0) .. controls (3.31,0.3) and (6.95,1.4) .. (10.93,3.29)   ;
\draw   (193.47,167.67) .. controls (193.47,158.86) and (200.45,151.71) .. (209.06,151.71) .. controls (217.67,151.71) and (224.65,158.86) .. (224.65,167.67) .. controls (224.65,176.48) and (217.67,183.63) .. (209.06,183.63) .. controls (200.45,183.63) and (193.47,176.48) .. (193.47,167.67) -- cycle ;
\draw    (378.31,199.13) .. controls (338.38,178.47) and (326.55,175.51) .. (294.7,198.08) ;
\draw [shift={(293.24,199.13)}, rotate = 324.27] [color={rgb, 255:red, 0; green, 0; blue, 0 }  ][line width=0.75]    (10.93,-3.29) .. controls (6.95,-1.4) and (3.31,-0.3) .. (0,0) .. controls (3.31,0.3) and (6.95,1.4) .. (10.93,3.29)   ;
\draw    (231.33,157.18) -- (258.49,143.5) ;
\draw [shift={(260.28,142.6)}, rotate = 153.25] [color={rgb, 255:red, 0; green, 0; blue, 0 }  ][line width=0.75]    (10.93,-3.29) .. controls (6.95,-1.4) and (3.31,-0.3) .. (0,0) .. controls (3.31,0.3) and (6.95,1.4) .. (10.93,3.29)   ;
\draw   (266.51,137.58) .. controls (266.51,128.77) and (273.49,121.62) .. (282.1,121.62) .. controls (290.71,121.62) and (297.69,128.77) .. (297.69,137.58) .. controls (297.69,146.39) and (290.71,153.54) .. (282.1,153.54) .. controls (273.49,153.54) and (266.51,146.39) .. (266.51,137.58) -- cycle ;
\draw    (304.37,136.21) -- (335.33,136.21) ;
\draw [shift={(337.33,136.21)}, rotate = 180] [color={rgb, 255:red, 0; green, 0; blue, 0 }  ][line width=0.75]    (10.93,-3.29) .. controls (6.95,-1.4) and (3.31,-0.3) .. (0,0) .. controls (3.31,0.3) and (6.95,1.4) .. (10.93,3.29)   ;
\draw   (347.57,136.67) .. controls (347.57,127.86) and (354.55,120.71) .. (363.16,120.71) .. controls (371.77,120.71) and (378.75,127.86) .. (378.75,136.67) .. controls (378.75,145.48) and (371.77,152.63) .. (363.16,152.63) .. controls (354.55,152.63) and (347.57,145.48) .. (347.57,136.67) -- cycle ;
\draw    (294.13,118.89) -- (317.45,101.83) ;
\draw [shift={(319.07,100.65)}, rotate = 143.83] [color={rgb, 255:red, 0; green, 0; blue, 0 }  ][line width=0.75]    (10.93,-3.29) .. controls (6.95,-1.4) and (3.31,-0.3) .. (0,0) .. controls (3.31,0.3) and (6.95,1.4) .. (10.93,3.29)   ;
\draw   (314.61,81.96) .. controls (314.61,73.14) and (321.59,66) .. (330.2,66) .. controls (338.81,66) and (345.79,73.14) .. (345.79,81.96) .. controls (345.79,90.77) and (338.81,97.91) .. (330.2,97.91) .. controls (321.59,97.91) and (314.61,90.77) .. (314.61,81.96) -- cycle ;
\draw    (341.34,99.74) -- (356.05,116.47) ;
\draw [shift={(357.37,117.98)}, rotate = 228.68] [color={rgb, 255:red, 0; green, 0; blue, 0 }  ][line width=0.75]    (10.93,-3.29) .. controls (6.95,-1.4) and (3.31,-0.3) .. (0,0) .. controls (3.31,0.3) and (6.95,1.4) .. (10.93,3.29)   ;
\draw   (425.07,138.49) .. controls (425.07,129.68) and (432.05,122.53) .. (440.66,122.53) .. controls (449.27,122.53) and (456.25,129.68) .. (456.25,138.49) .. controls (456.25,147.3) and (449.27,154.45) .. (440.66,154.45) .. controls (432.05,154.45) and (425.07,147.3) .. (425.07,138.49) -- cycle ;
\draw    (385.43,138.04) -- (416.39,138.04) ;
\draw [shift={(418.39,138.04)}, rotate = 180] [color={rgb, 255:red, 0; green, 0; blue, 0 }  ][line width=0.75]    (10.93,-3.29) .. controls (6.95,-1.4) and (3.31,-0.3) .. (0,0) .. controls (3.31,0.3) and (6.95,1.4) .. (10.93,3.29)   ;
\draw    (463.82,138.04) -- (494.78,138.04) ;
\draw [shift={(496.78,138.04)}, rotate = 180] [color={rgb, 255:red, 0; green, 0; blue, 0 }  ][line width=0.75]    (10.93,-3.29) .. controls (6.95,-1.4) and (3.31,-0.3) .. (0,0) .. controls (3.31,0.3) and (6.95,1.4) .. (10.93,3.29)   ;
\draw   (507.02,138.49) .. controls (507.02,129.68) and (514,122.53) .. (522.61,122.53) .. controls (531.22,122.53) and (538.2,129.68) .. (538.2,138.49) .. controls (538.2,147.3) and (531.22,154.45) .. (522.61,154.45) .. controls (514,154.45) and (507.02,147.3) .. (507.02,138.49) -- cycle ;
\draw    (545.82,139.04) -- (576.78,139.04) ;
\draw [shift={(578.78,139.04)}, rotate = 180] [color={rgb, 255:red, 0; green, 0; blue, 0 }  ][line width=0.75]    (10.93,-3.29) .. controls (6.95,-1.4) and (3.31,-0.3) .. (0,0) .. controls (3.31,0.3) and (6.95,1.4) .. (10.93,3.29)   ;
\draw   (589.02,139.49) .. controls (589.02,130.68) and (596,123.53) .. (604.61,123.53) .. controls (613.22,123.53) and (620.2,130.68) .. (620.2,139.49) .. controls (620.2,148.3) and (613.22,155.45) .. (604.61,155.45) .. controls (596,155.45) and (589.02,148.3) .. (589.02,139.49) -- cycle ;
\draw    (592,117) -- (559.54,90.27) ;
\draw [shift={(558,89)}, rotate = 39.47] [color={rgb, 255:red, 0; green, 0; blue, 0 }  ][line width=0.75]    (10.93,-3.29) .. controls (6.95,-1.4) and (3.31,-0.3) .. (0,0) .. controls (3.31,0.3) and (6.95,1.4) .. (10.93,3.29)   ;
\draw   (522.61,79.96) .. controls (522.61,71.14) and (529.59,64) .. (538.2,64) .. controls (546.81,64) and (553.79,71.14) .. (553.79,79.96) .. controls (553.79,88.77) and (546.81,95.91) .. (538.2,95.91) .. controls (529.59,95.91) and (522.61,88.77) .. (522.61,79.96) -- cycle ;
\draw    (593,156) .. controls (586.11,187.52) and (326.95,162.77) .. (290.51,155.33) ;
\draw [shift={(289,155)}, rotate = 13.13] [color={rgb, 255:red, 0; green, 0; blue, 0 }  ][line width=0.75]    (10.93,-3.29) .. controls (6.95,-1.4) and (3.31,-0.3) .. (0,0) .. controls (3.31,0.3) and (6.95,1.4) .. (10.93,3.29)   ;
\draw    (519,88) -- (447.8,122.14) ;
\draw [shift={(446,123)}, rotate = 334.38] [color={rgb, 255:red, 0; green, 0; blue, 0 }  ][line width=0.75]    (10.93,-3.29) .. controls (6.95,-1.4) and (3.31,-0.3) .. (0,0) .. controls (3.31,0.3) and (6.95,1.4) .. (10.93,3.29)   ;
\draw    (561,76) -- (597,76.95) ;
\draw [shift={(599,77)}, rotate = 181.51] [color={rgb, 255:red, 0; green, 0; blue, 0 }  ][line width=0.75]    (10.93,-3.29) .. controls (6.95,-1.4) and (3.31,-0.3) .. (0,0) .. controls (3.31,0.3) and (6.95,1.4) .. (10.93,3.29)   ;
\draw   (603.02,75.49) .. controls (603.02,66.68) and (610,59.53) .. (618.61,59.53) .. controls (627.22,59.53) and (634.2,66.68) .. (634.2,75.49) .. controls (634.2,84.3) and (627.22,91.45) .. (618.61,91.45) .. controls (610,91.45) and (603.02,84.3) .. (603.02,75.49) -- cycle ;
\draw   (561.13,28.49) .. controls (561.13,19.68) and (568.11,12.53) .. (576.72,12.53) .. controls (585.33,12.53) and (592.31,19.68) .. (592.31,28.49) .. controls (592.31,37.3) and (585.33,44.45) .. (576.72,44.45) .. controls (568.11,44.45) and (561.13,37.3) .. (561.13,28.49) -- cycle ;
\draw    (605,66) -- (585.32,43.51) ;
\draw [shift={(584,42)}, rotate = 48.81] [color={rgb, 255:red, 0; green, 0; blue, 0 }  ][line width=0.75]    (10.93,-3.29) .. controls (6.95,-1.4) and (3.31,-0.3) .. (0,0) .. controls (3.31,0.3) and (6.95,1.4) .. (10.93,3.29)   ;
\draw    (564,42) -- (551.04,63.29) ;
\draw [shift={(550,65)}, rotate = 301.33] [color={rgb, 255:red, 0; green, 0; blue, 0 }  ][line width=0.75]    (10.93,-3.29) .. controls (6.95,-1.4) and (3.31,-0.3) .. (0,0) .. controls (3.31,0.3) and (6.95,1.4) .. (10.93,3.29)   ;

\draw (35.2,210.23) node [anchor=north west][inner sep=0.75pt]   [align=left] {0};
\draw (117.6,211.15) node [anchor=north west][inner sep=0.75pt]   [align=left] {33};
\draw (205.85,210.23) node [anchor=north west][inner sep=0.75pt]   [align=left] {11};
\draw (286.85,210.23) node [anchor=north west][inner sep=0.75pt]   [align=left] {37};
\draw (372.73,211.15) node [anchor=north west][inner sep=0.75pt]   [align=left] {12};
\draw (464.92,211.15) node [anchor=north west][inner sep=0.75pt]   [align=left] {4};
\draw (548.66,210.23) node [anchor=north west][inner sep=0.75pt]   [align=left] {1};
\draw (200.45,159.17) node [anchor=north west][inner sep=0.75pt]   [align=left] {44};
\draw (273.49,129.99) node [anchor=north west][inner sep=0.75pt]   [align=left] {48};
\draw (355.44,129.08) node [anchor=north west][inner sep=0.75pt]   [align=left] {16};
\draw (321.59,74.37) node [anchor=north west][inner sep=0.75pt]   [align=left] {49};
\draw (434.27,130.9) node [anchor=north west][inner sep=0.75pt]   [align=left] {5};
\draw (514.9,130.9) node [anchor=north west][inner sep=0.75pt]   [align=left] {35};
\draw (596.9,131.9) node [anchor=north west][inner sep=0.75pt]   [align=left] {45};
\draw (528.59,73.37) node [anchor=north west][inner sep=0.75pt]   [align=left] {15};
\draw (610.9,67.9) node [anchor=north west][inner sep=0.75pt]   [align=left] {38};
\draw (569,20.9) node [anchor=north west][inner sep=0.75pt]   [align=left] {46};

\end{tikzpicture}
\caption{The subgraph of the multiplication transducer for $m=100$ that encodes all representations of $100$ as a quotient of sums of distinct powers of 3.}\label{digraph for 100}
\end{figure}
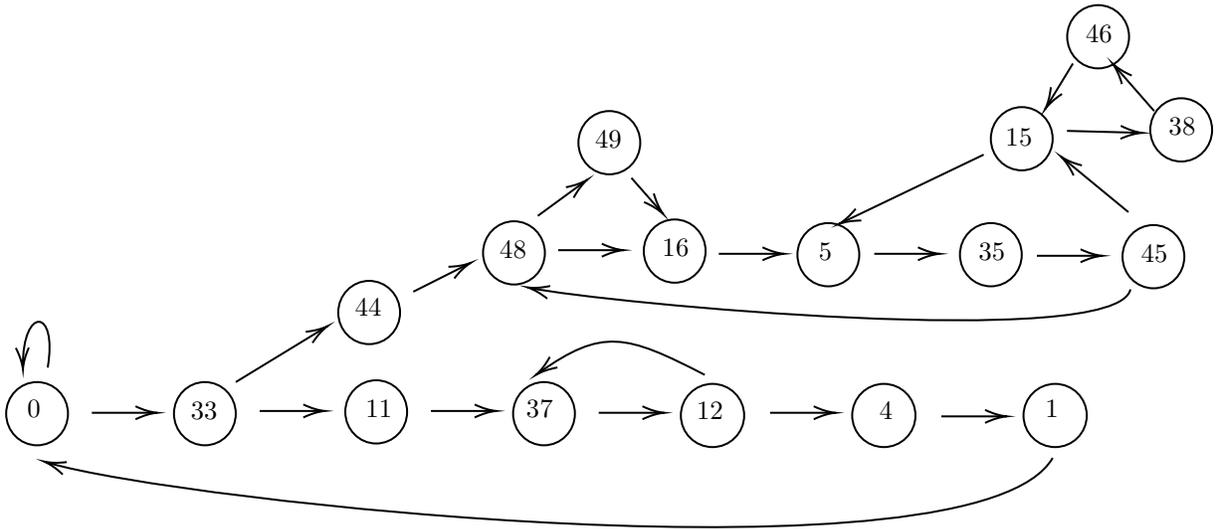

Continuing with the algorithm, our next carry value is $37=3^3+3^2+3^0$, a sum of distinct powers of 3.  According to the algorithm, we now record  $100 = \frac{[{\bf1101}001]_3}{\hspace{1.81em}[101]_3}$, where the rightmost three digits in both the numerator and denominator come from the three steps carrying 0, then $33$, then $11$, and the remaining bold digits in the numerator come from appending $[1101]_3$, the base 3 expansion of 37, to $p$.   We have $100=\frac{3^6+3^5+3^3+1}{3^2+1}$, and  with $p(x)=x^6+x^5+x^3+1$ and $q(x)=x^2+1$, this is $100=\frac{p(3)}{q(3)}$.  Is this a universal representation or a local representation?  Since $p(x)=(x^4+x^3-x^2+1)q(x)$, this representation is universal.

As soon as we reach a carry value that is a sum of distinct powers of 3, in this case the carry value 37, the algorithm takes the shortest possible path back to 0 by always choosing to step down when the carry value is congruent to $0\bmod 3$. We could obtain other indecomposable representations of $100$ by deviating from the algorithm and instead choosing to step up from 12, possibly repeatedly, and each step up from $12$ would create a two-step loop in the journey back to 0.  In fact, we could obtain infinitely many indecomposable representations in this fashion.  Would they be all universal, all local, or some of each?

\begin{theorem}\label{100 indecomposable}
Every indecomposable representation of $100$ is universal.  Moreover, for any indecomposable representation of $100$ as $\frac{p(3)}{q(3)}$, the quotient $\frac{p(x)}{q(x)}$ is $x^4+x^3-x^2+1$.
\end{theorem}

\begin{proof}
As noted above, taking a step up from $33$ does not lead to a representation of $100$.  Also noted above, if we step down the first time we encounter a carry value of $12$, we obtain a universal representation.  Instead, suppose that when we first encounter a carry value of $12$, we step up from $12$ to $37$, follow the two-step loop exactly once, and then continue as usual. This yields the recorded quotient 
\[100 =  \frac{[{110}\underline{11}1001]_3}{\hspace{1.5em}[\underline{01}0101]_3}= \frac{3^8 + 3^7 +3^5 + 3^4+ 3^3+1}{3^4+3^2+1}.\]
If we step up from $12$ to $37$ exactly twice and then continue, we record 
\[100 =  \frac{[{110}\underline{1111}1001]_3}{\hspace{1.5em}[\underline{0101}0101]_3}= \frac{3^{10} + 3^9 +3^7+3^6+3^5 + 3^4+ 3^3+1}{3^6+3^4+3^2+1}.\]
To generalize, after stepping up from 12 to $37$ exactly $j$ times, which means taking the two-step loop in Figure \ref{digraph for 100} exactly $j$ times, we have
\[
100=\frac{3^{4+2j+2}+3^{4+2j+1}+\displaystyle\sum_{k=4}^{4+2j-1}3^k+3^3+1}{\displaystyle\sum_{i=1}^j3^{2+2i}+3^2+1}.
\]

For any $j\geq 0$, let $p_j(x)=x^{4+2j+2}+x^{4+2j+1}+\sum\limits_{k=4}^{4+2j-1}x^k+x^3+1$ and $q_j(x)=\sum\limits_{i=1}^jx^{2+2i}+x^2+1$.  We will show by induction that $p_j(x)=(x^4+x^3-x^2+1)q_j(x)$ for any $j\geq 0$.  For the base case, note that when $j=0$, we have $p_0(x)=x^6+x^5+x^3+1$ and $q_0(x)=x^2+1$, so $p_0(x)=(x^4+x^3-x^2+1)q_0(x)$.  Suppose that $t$ is a non-negative integer with $p_t(x)=(x^4+x^3-x^2+1)q_t(x)$.  We will show that $p_{t+1}(x)=(x^4+x^3-x^2+1)q_{t+1}(x)$ to complete the proof.  First note that
\begin{align*}
(x^4+x^3-x^2+1)q_{t+1}(x)&=(x^4+x^3-x^2+1)\left(\displaystyle\sum_{i=1}^{t+1}x^{2+2i}+x^2+1\right)\\
&=(x^4+x^3-x^2+1)\left(x^{2+2(t+1)}+\displaystyle\sum_{i=1}^{t}x^{2+2i}+x^2+1\right)\\
&=(x^4+x^3-x^2+1)(x^{2+2t+2})+p_t(x).\\
\end{align*}
Now observe that $p_t(x)=x^{4+2t+2}+x^{4+2t+1}+\sum\limits_{k=4}^{4+2t-1}x^k+x^3+1$, while
\begin{align*}
p_{t+1}(x)&=x^{4+2(t+1)+2}+x^{4+2(t+1)+1}+\displaystyle\sum_{k=4}^{4+2(t+1)-1}x^k+x^3+1\\
&=x^{4+2t+4}+x^{4+2t+3}+x^{4+2t+1}+x^{4+2t}+\displaystyle\sum_{k=4}^{4+2t-1}x^k+x^3+1\\
&=x^{4+2t+4}+x^{4+2t+3}+x^{4+2t+1}+x^{4+2t}+\displaystyle\sum_{k=4}^{4+2t-1}x^k+x^3+1+(x^{4+2t+2}-x^{4+2t+2})\\
&=x^{4+2t+4}+x^{4+2t+3}-x^{4+2t+2}+x^{4+2t}+p_t(x)\\
&=x^{4+2t}(x^4+x^3-x^2+1)+p_t(x).
\end{align*}
Thus $p_{t+1}(x)=(x^4+x^3-x^2+1)q_{t+1}(x)$, and our proof by induction is complete.  

Since we have considered all indecomposable representations of $100$, we conclude that every indecomposable representation of $100$ is universal.
\end{proof}

\begin{theorem}\label{EveryRepOf100Universal}
Every representation of $100$ is universal.  Moreover, for any representation of $100$ as $\frac{p(3)}{q(3)}$, the quotient $\frac{p(x)}{q(x)}$ is $x^4+x^3-x^2+1$.  This includes both indecomposable and decomposable representations.
\end{theorem}

\begin{proof}
We have addressed the indecomposable representations of $100$ in Theorem \ref{100 indecomposable}, so it remains to consider the decomposable representations of $100$.  Suppose we have a representation of $100$ resulting from the concatenation of two indecomposable walks through the digraph with no loops at 0 occurring between the two walks. On the first indecomposable walk, we step up from $12$ to $37$ exactly $j$ times, and on the second indecomposable walk, we step up from $12$ to $37$ exactly $k$ times.  

The contributions to $p(x)$ and $q(x)$ coming from this first indecomposable walk are $p_j(x)$ and $q_j(x)$, as defined in the proof of Theorem \ref{100 indecomposable}, while the contributions coming from the second indecomposable walk are $x^{2j+7}p_k(x)$ and $x^{2j+7}q_k(x)$.  Note that there are a total of $2j+7$ steps in the first indecomposable walk, so the second indecomposable walk begins recording powers of 3 starting with $3^{2j+7}$.  Thus, for the representation of $100$ resulting from two indecomposable walks, we have $p(x)=x^{2j+7}p_k(x)+p_j(x)$ and $q(x)=x^{2j+7}q_k(x)+q_j(x)$.

Note that 
\begin{align*}
(x^4+x^3-x^2+1)q(x)&=(x^4+x^3-x^2+1)\left(x^{2j+7}q_k(x)+q_j(x)\right)\\
&=x^{2j+7}(x^4+x^3-x^2+1)q_k(x)+(x^4+x^3-x^2+1)q_j(x)\\
&=x^{2j+7}p_k(x)+p_j(x)\\
&=p(x).
\end{align*}
Thus $q(x)\mid p(x)$, and this representation of $100$ must be universal.  

What if a representation results from concatenating not just two indecomposable walks through the digraph but $n$ indecomposable walks with no loops at 0 between any two walks?  Let $i_1, i_2,\ldots i_n$ be the non-negative integers such that $i_h$ is the number of times we step up from $12$ to $37$ on the $h$th indecomposable walk through the digraph.  Then we have
\begin{align*}
p(x)&=x^{2(i_1+i_2+\cdots+i_{n-1})+7(n-1)}p_{i_n}(x)+x^{2(i_1+i_2+\cdots+i_{n-2})+7(n-2)}p_{i_{n-1}}(x)+\cdots\\
&\qquad+x^{2(i_1+i_2)+7(2)}p_{i_3}(x)+x^{2i_1+7}p_{i_2}(x)+p_{i_1}(x)\\
&=\displaystyle\sum_{j=1}^n x^{2a_j+7(j-1)}p_{i_j}(x),
\end{align*}
where $a_j=\sum\limits_{k=1}^{j-1} i_k$, and
\begin{align*}
q(x)&=x^{2(i_1+i_2+\cdots+i_{n-1})+7(n-1)}q_{i_n}(x)+x^{2(i_1+i_2+\cdots+i_{n-2})+7(n-2)}q_{i_{n-1}}(x)+\cdots\\
&\qquad+x^{2(i_1+i_2)+7(2)}q_{i_3}(x)+x^{2i_1+7}q_{i_2}(x)+q_{i_1}(x)\\
&=\displaystyle\sum_{j=1}^n x^{2a_j+7(j-1)}q_{i_j}(x).
\end{align*}
Each pair of polynomials $p_{i_h}(x)$  and $q_{i_h}(x)$ gives an indecomposable representation of $100$ as $100=\frac{p_{i_h}(3)}{q_{i_h}(3)}$, so from Theorem \ref{100 indecomposable}, we have $p_{i_h}(x) = (x^4 + x^3 -x^2+1)q_{i_h}(x)$ for each $i_h$.  Then we see that
\begin{align*}
(x^4+x^3-x^2+1)q(x)&=(x^4+x^3-x^2+1)\displaystyle\sum_{j=1}^n x^{2a_j+7(j-1)}q_{i_j}(x)\\
&=\displaystyle\sum_{j=1}^n x^{2a_j+7(j-1)}(x^4+x^3-x^2+1)q_{i_j}(x)\\
&=\displaystyle\sum_{j=1}^n x^{2a_j+7(j-1)}p_{i_j}(x)\\
&=p(x).
\end{align*}
Thus $q(x)\mid p(x)$, and this representation of $100$ must be universal.

We conclude by proving that adding loops at 0 between any two indecomposable walks results in new representations of 100 which are also universal.  Consider $p(x)=\sum_hx^{b_h}p_{i_h}(x)$ and $q(x)=\sum_hx^{b_h}q_{i_h}(x)$, where $i_1, i_2,\ldots i_n$ are as above and $b_1, b_2,\ldots b_n$ are non-negative integers. Then by Lemma \ref{ForTheorem2.5}, we see that $p(x)=(x^4+x^3-x^2+1)q(x)$ as well.  Therefore, since adding loops at 0 changes only the powers $b_h$ of $x$ in $p(x)$ and $q(x)$, we see that representations of 100 including loops at 0 are still universal. As an example, if we take $\ell$ loops at 0 between walk $i_d$ and walk $i_{d+1}$, the resulting polynomials $p(x)$ and $q(x)$ yield
\begin{align*}
(x^4+x^3-x^2+1)q(x)&=(x^4+x^3-x^2+1)\left[\displaystyle\sum_{j=1}^d x^{2a_j+7(j-1)}q_{i_j}(x) + \displaystyle\sum_{j=d+1}^n x^{2a_j+7(j-1)+ \ell}q_{i_j}(x) \right]\\
&=\displaystyle\sum_{j=1}^d x^{2a_j+7(j-1)}(x^4+x^3-x^2+1)q_{i_j}(x) \\
&\qquad+\displaystyle\sum_{j=d+1}^n x^{2a_j+7(j-1)+ \ell}(x^4+x^3-x^2+1)q_{i_j}(x) \\
&=\displaystyle\sum_{j=1}^d x^{2a_j+7(j-1)}p_{i_j}(x) + \displaystyle\sum_{j=d+1}^n x^{2a_j+7(j-1)+ \ell}p_{i_j}(x)\\
&=p(x).
\end{align*}

Hence all representations of $100$ are universal.
\end{proof}


\section{Integers with local representations}\label{local representations}

As we have seen in Examples \ref{31 example} and \ref{37 example}, some integers $m\equiv1\pmod{3}$ have local representations in addition to a trivial universal representation.  In this section, we identify several other examples of integers with local representations. Specifically, in Section \ref{OnlyLocalSection}, we restrict our attention to integers $m$ with only local representations, and in Section \ref{64Section}, we show that $m=64$ has both local representations and nontrivial universal representations.

\subsection{Integers with only local representations}\label{OnlyLocalSection}

Some integers $m\equiv1\pmod{3}$ have only local representations.  Here we prove that all representations of $m=22$ and $m=34$ are local, and we list many more integers $m$ whose representations are all local.

\begin{theorem}\label{22allhard}
There do not exist $p(x),q(x) \in \mathcal P$ with $q(x) \ | \ p(x)$ such that $22 = \frac{p(3)}{q(3)}$; that is, all representations of $22$ are local, including both decomposable and indecomposable representations.
\end{theorem}
\begin{proof}
Suppose to the contrary that $22 = \frac{p(3)}{q(3)}$, where $p(x),q(x)\in\mathcal{P}$ and $g(x):=\frac{p(x)}{q(x)}\in\mathbb{Z}[x]$. Let $r=\deg p(x)$ and $s=\deg q(x)$.  By Theorem \ref{T:feasible},
\[
\frac 23 \cdot 3^{r-s} < 22 < \frac 32 \cdot 3^{r-s},
\]
which implies $r -s =
\deg g(x) = 3$.  Since the constant terms and the coefficients of $x^r$ in $p(x)$ and
$x^s$ in $q(x)$ are 1, we may then write $g(x) = 1 + e_1 x + e_2 x^2 +
x^3$. It is now convenient to take $p(x) = 1 + \sum\limits_{i=1}^{r-1}  c_i x^i +x^r$ and $q(x) = 1 + \sum\limits_{i=1}^{s-1} d_i x^i+x^s$, where $c_i, d_i \in \{0,1\}$.  The equation $p(x) = g(x)q(x)$ thus implies that
\begin{align*}
c_1 &= d_1 + e_1, \\
c_2 &= d_2 + d_1e_1 + e_2, \\
c_3 &= d_3 + d_2e_1 + d_1e_2 + 1.
\end{align*}
Since $22 = g(3) = 1 + 3e_1 + 9e_2 + 27$, we have $e_1 + 3e_2 = -2$,
and so $e_1 \equiv 1 \pmod 3$. But $e_1 = c_1 - d_1$ and $c_1,d_1 \in
\{0,1\}$, so we must have that $c_1 = 1, d_1 = 0,$ and $ e_1 = 1$.  Thus $3e_2 =
-3$, so $e_2 = -1$ and $g(x) = x^3 - x^2 + x + 1$.  Note that
$g(3) = 22$. The remaining two equations become
\begin{align*}
c_2 &= d_2 + 0\cdot 1 + (-1) = d_2 - 1,\\
c_3 &= d_3 + d_2\cdot 1 + 0\cdot(-1) + 1 = d_3 + d_2 + 1.
\end{align*}
It follows from the first of these two equations that $c_2 = 0$ and
$d_2 = 1$. It then follows from the second equation that $c_3 = d_3 + 2$, which
is impossible. 
\end{proof}

We can implement the same approach to prove that all representations of $34$ are local, or we can instead use Theorem \ref{22allhard} and prove the result regarding $34$ as a corollary.  Here we implement the latter approach.

\begin{corollary}\label{34allhard}
There do not exist $p(x),q(x) \in \mathcal P$ with $q(x) \ | \ p(x)$ such that $34 = \frac{p(3)}{q(3)}$; that is, all representations of $34$ are local, including both decomposable and indecomposable representations.

\end{corollary}
\begin{proof}
Let $p(x),q(x),r$, and $s$ be as in the proof of Theorem \ref{22allhard}, and let $h(x):=\frac{p(x)}{q(x)}\in\mathbb{Z}[x]$.  Again,
\[
\frac 23 \cdot 3^{r-s} < 34 < \frac 32 \cdot 3^{r-s},
\]
which implies $r -s =
\deg h(x) = 3$, and  we will continue following the same procedure as in the proof of Theorem \ref{22allhard}. Now $34 = h(3) = 1 + 3e_1 + 9e_2 + 27$, so 
$e_1 + 3e_2 = 2$ and $e_1 \equiv 2 \pmod 3$. But $e_1 = c_1 - d_1$ and $c_1,d_1 \in
\{0,1\}$, so we have that $c_1 = 0$, $d_1 = 1$, and $e_1 = -1$.  Thus $e_2 = 1$ and 
$h(x) = x^3 +x^2 - x + 1$. 

At this point, our argument diverges from that in the proof of Theorem \ref{22allhard}. Write $p(x)=\sum\limits_{k=0}^m x^{a_k}$ and $q(x)=\sum\limits_{\ell=0}^n x^{b_\ell}$.  We have $q(x)(x^3 +x^2 - x + 1) = p(x)$. Now
reverse the polynomials by taking $x \mapsto 1/x$ and multiplying $q(x)$ by $x^s$, $h(x)$ by $x^3$,
and $p(x)$ by $x^r$ to obtain
\begin{equation}\label{eq for 34}
\left( \sum_{\ell=0}^n x^{s-b_\ell}\right) (x^3-x^2+x+1) = \sum_{k=0}^m x^{r-a_k} . 
\end{equation}
Note that $x^3 h\left(\frac{1}{x}\right)=g(x)$, where $g(x)$ is as in the proof of Theorem \ref{22allhard}. Since the polynomials $x^s q\left(\frac{1}{x}\right)$ and $x^r p\left(\frac{1}{x}\right)$ are still in $\mathcal P$, we have already seen in the proof of Theorem \ref{22allhard} that Equation \eqref{eq for 34} cannot be satisfied, and we have a contradiction.
\end{proof}

The same type of proof by contradiction that we used to prove Theorem \ref{22allhard} by examining the coefficients of $p(x)$ and $q(x)$ can also be used to show that the following numbers have only local representations: 34, 58, 67, 97, 103, 106, 115, 175, 178, 184, 193, 199, 202, 205, 208, 214, 229, 232, 238, 259, 265, 277, 286, 295, 298, 304, 307, 310, 313, 331, 340. Additionally, the procedure can be applied to test larger integers, but the computations become increasingly tedious.

\subsection{Integers with both nontrivial universal representations and local representations}\label{64Section} Here we consider $m=64$ as an example of an integer that has both nontrivial universal representations and local representations.  See the digraph for $m=64$ in Figure \ref{digraph for 64} for reference.

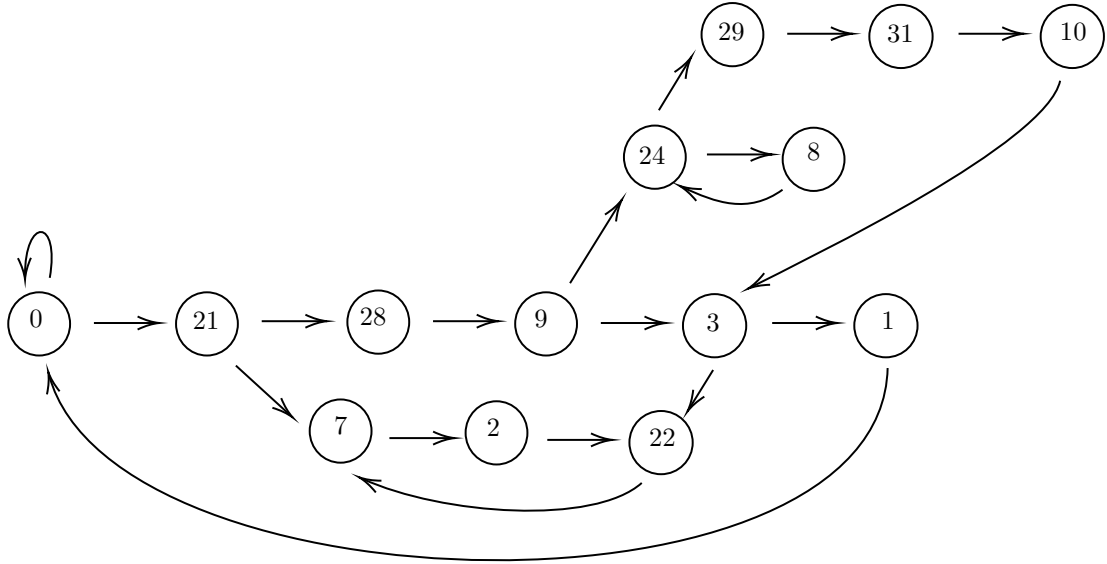
\begin{figure}
\tikzset{every picture/.style={line width=0.75pt}} 

\begin{tikzpicture}[x=0.75pt,y=0.75pt,yscale=-1,xscale=1]

\draw   (26,218.73) .. controls (26,209.92) and (32.98,202.78) .. (41.59,202.78) .. controls (50.2,202.78) and (57.18,209.92) .. (57.18,218.73) .. controls (57.18,227.55) and (50.2,234.69) .. (41.59,234.69) .. controls (32.98,234.69) and (26,227.55) .. (26,218.73) -- cycle ;
\draw    (69.2,218.28) -- (100.16,218.28) ;
\draw [shift={(102.16,218.28)}, rotate = 180] [color={rgb, 255:red, 0; green, 0; blue, 0 }  ][line width=0.75]    (10.93,-3.29) .. controls (6.95,-1.4) and (3.31,-0.3) .. (0,0) .. controls (3.31,0.3) and (6.95,1.4) .. (10.93,3.29)   ;
\draw   (110.62,218.73) .. controls (110.62,209.92) and (117.6,202.78) .. (126.21,202.78) .. controls (134.82,202.78) and (141.8,209.92) .. (141.8,218.73) .. controls (141.8,227.55) and (134.82,234.69) .. (126.21,234.69) .. controls (117.6,234.69) and (110.62,227.55) .. (110.62,218.73) -- cycle ;
\draw    (153.83,217.37) -- (184.79,217.37) ;
\draw [shift={(186.79,217.37)}, rotate = 180] [color={rgb, 255:red, 0; green, 0; blue, 0 }  ][line width=0.75]    (10.93,-3.29) .. controls (6.95,-1.4) and (3.31,-0.3) .. (0,0) .. controls (3.31,0.3) and (6.95,1.4) .. (10.93,3.29)   ;
\draw   (197.03,217.82) .. controls (197.03,209.01) and (204.01,201.87) .. (212.62,201.87) .. controls (221.23,201.87) and (228.21,209.01) .. (228.21,217.82) .. controls (228.21,226.64) and (221.23,233.78) .. (212.62,233.78) .. controls (204.01,233.78) and (197.03,226.64) .. (197.03,217.82) -- cycle ;
\draw    (240.23,217.37) -- (271.19,217.37) ;
\draw [shift={(273.19,217.37)}, rotate = 180] [color={rgb, 255:red, 0; green, 0; blue, 0 }  ][line width=0.75]    (10.93,-3.29) .. controls (6.95,-1.4) and (3.31,-0.3) .. (0,0) .. controls (3.31,0.3) and (6.95,1.4) .. (10.93,3.29)   ;
\draw   (281.66,218.73) .. controls (281.66,209.92) and (288.64,202.78) .. (297.24,202.78) .. controls (305.85,202.78) and (312.83,209.92) .. (312.83,218.73) .. controls (312.83,227.55) and (305.85,234.69) .. (297.24,234.69) .. controls (288.64,234.69) and (281.66,227.55) .. (281.66,218.73) -- cycle ;
\draw    (324.86,218.28) -- (355.82,218.28) ;
\draw [shift={(357.82,218.28)}, rotate = 180] [color={rgb, 255:red, 0; green, 0; blue, 0 }  ][line width=0.75]    (10.93,-3.29) .. controls (6.95,-1.4) and (3.31,-0.3) .. (0,0) .. controls (3.31,0.3) and (6.95,1.4) .. (10.93,3.29)   ;
\draw   (366.28,219.65) .. controls (366.28,210.83) and (373.42,203.69) .. (382.23,203.69) .. controls (391.04,203.69) and (398.19,210.83) .. (398.19,219.65) .. controls (398.19,228.46) and (391.04,235.6) .. (382.23,235.6) .. controls (373.42,235.6) and (366.28,228.46) .. (366.28,219.65) -- cycle ;
\draw    (411.27,218.28) -- (442.22,218.28) ;
\draw [shift={(444.22,218.28)}, rotate = 180] [color={rgb, 255:red, 0; green, 0; blue, 0 }  ][line width=0.75]    (10.93,-3.29) .. controls (6.95,-1.4) and (3.31,-0.3) .. (0,0) .. controls (3.31,0.3) and (6.95,1.4) .. (10.93,3.29)   ;
\draw   (452.69,219.65) .. controls (452.69,210.83) and (459.83,203.69) .. (468.64,203.69) .. controls (477.45,203.69) and (484.59,210.83) .. (484.59,219.65) .. controls (484.59,228.46) and (477.45,235.6) .. (468.64,235.6) .. controls (459.83,235.6) and (452.69,228.46) .. (452.69,219.65) -- cycle ;
\draw    (469.5,241) .. controls (466.52,374.33) and (79.4,365.1) .. (46.62,244.72) ;
\draw [shift={(46.15,242.9)}, rotate = 76.48] [color={rgb, 255:red, 0; green, 0; blue, 0 }  ][line width=0.75]    (10.93,-3.29) .. controls (6.95,-1.4) and (3.31,-0.3) .. (0,0) .. controls (3.31,0.3) and (6.95,1.4) .. (10.93,3.29)   ;
\draw    (46.93,195.48) .. controls (52.17,165.1) and (34.31,164.49) .. (34.43,194.52) ;
\draw [shift={(34.46,196.39)}, rotate = 268.4] [color={rgb, 255:red, 0; green, 0; blue, 0 }  ][line width=0.75]    (10.93,-3.29) .. controls (6.95,-1.4) and (3.31,-0.3) .. (0,0) .. controls (3.31,0.3) and (6.95,1.4) .. (10.93,3.29)   ;
\draw    (140.8,239.78) -- (168.02,264.65) ;
\draw [shift={(169.5,266)}, rotate = 222.42] [color={rgb, 255:red, 0; green, 0; blue, 0 }  ][line width=0.75]    (10.93,-3.29) .. controls (6.95,-1.4) and (3.31,-0.3) .. (0,0) .. controls (3.31,0.3) and (6.95,1.4) .. (10.93,3.29)   ;
\draw    (416.5,151) .. controls (399.13,163.55) and (382.69,157.47) .. (366.29,149.84) ;
\draw [shift={(364.5,149)}, rotate = 25.2] [color={rgb, 255:red, 0; green, 0; blue, 0 }  ][line width=0.75]    (10.93,-3.29) .. controls (6.95,-1.4) and (3.31,-0.3) .. (0,0) .. controls (3.31,0.3) and (6.95,1.4) .. (10.93,3.29)   ;
\draw    (309.33,198.18) -- (334.45,157.7) ;
\draw [shift={(335.5,156)}, rotate = 121.82] [color={rgb, 255:red, 0; green, 0; blue, 0 }  ][line width=0.75]    (10.93,-3.29) .. controls (6.95,-1.4) and (3.31,-0.3) .. (0,0) .. controls (3.31,0.3) and (6.95,1.4) .. (10.93,3.29)   ;
\draw   (336.51,134.58) .. controls (336.51,125.77) and (343.49,118.62) .. (352.1,118.62) .. controls (360.71,118.62) and (367.69,125.77) .. (367.69,134.58) .. controls (367.69,143.39) and (360.71,150.54) .. (352.1,150.54) .. controls (343.49,150.54) and (336.51,143.39) .. (336.51,134.58) -- cycle ;
\draw    (378.37,133.21) -- (409.33,133.21) ;
\draw [shift={(411.33,133.21)}, rotate = 180] [color={rgb, 255:red, 0; green, 0; blue, 0 }  ][line width=0.75]    (10.93,-3.29) .. controls (6.95,-1.4) and (3.31,-0.3) .. (0,0) .. controls (3.31,0.3) and (6.95,1.4) .. (10.93,3.29)   ;
\draw   (416.57,135.67) .. controls (416.57,126.86) and (423.55,119.71) .. (432.16,119.71) .. controls (440.77,119.71) and (447.75,126.86) .. (447.75,135.67) .. controls (447.75,144.48) and (440.77,151.63) .. (432.16,151.63) .. controls (423.55,151.63) and (416.57,144.48) .. (416.57,135.67) -- cycle ;
\draw    (354.13,110.89) -- (369.46,85.71) ;
\draw [shift={(370.5,84)}, rotate = 121.34] [color={rgb, 255:red, 0; green, 0; blue, 0 }  ][line width=0.75]    (10.93,-3.29) .. controls (6.95,-1.4) and (3.31,-0.3) .. (0,0) .. controls (3.31,0.3) and (6.95,1.4) .. (10.93,3.29)   ;
\draw    (556.5,96) .. controls (549.61,127.52) and (432.1,184.27) .. (400.39,200.3) ;
\draw [shift={(399,201)}, rotate = 333.05] [color={rgb, 255:red, 0; green, 0; blue, 0 }  ][line width=0.75]    (10.93,-3.29) .. controls (6.95,-1.4) and (3.31,-0.3) .. (0,0) .. controls (3.31,0.3) and (6.95,1.4) .. (10.93,3.29)   ;
\draw   (178.03,272.82) .. controls (178.03,264.01) and (185.01,256.87) .. (193.62,256.87) .. controls (202.23,256.87) and (209.21,264.01) .. (209.21,272.82) .. controls (209.21,281.64) and (202.23,288.78) .. (193.62,288.78) .. controls (185.01,288.78) and (178.03,281.64) .. (178.03,272.82) -- cycle ;
\draw    (218.23,276.37) -- (249.19,276.37) ;
\draw [shift={(251.19,276.37)}, rotate = 180] [color={rgb, 255:red, 0; green, 0; blue, 0 }  ][line width=0.75]    (10.93,-3.29) .. controls (6.95,-1.4) and (3.31,-0.3) .. (0,0) .. controls (3.31,0.3) and (6.95,1.4) .. (10.93,3.29)   ;
\draw   (256.66,273.73) .. controls (256.66,264.92) and (263.64,257.78) .. (272.24,257.78) .. controls (280.85,257.78) and (287.83,264.92) .. (287.83,273.73) .. controls (287.83,282.55) and (280.85,289.69) .. (272.24,289.69) .. controls (263.64,289.69) and (256.66,282.55) .. (256.66,273.73) -- cycle ;
\draw    (297.86,277.28) -- (328.82,277.28) ;
\draw [shift={(330.82,277.28)}, rotate = 180] [color={rgb, 255:red, 0; green, 0; blue, 0 }  ][line width=0.75]    (10.93,-3.29) .. controls (6.95,-1.4) and (3.31,-0.3) .. (0,0) .. controls (3.31,0.3) and (6.95,1.4) .. (10.93,3.29)   ;
\draw   (339.28,278.65) .. controls (339.28,269.83) and (346.42,262.69) .. (355.23,262.69) .. controls (364.04,262.69) and (371.19,269.83) .. (371.19,278.65) .. controls (371.19,287.46) and (364.04,294.6) .. (355.23,294.6) .. controls (346.42,294.6) and (339.28,287.46) .. (339.28,278.65) -- cycle ;
\draw    (381.5,242) -- (369.55,261.3) ;
\draw [shift={(368.5,263)}, rotate = 301.76] [color={rgb, 255:red, 0; green, 0; blue, 0 }  ][line width=0.75]    (10.93,-3.29) .. controls (6.95,-1.4) and (3.31,-0.3) .. (0,0) .. controls (3.31,0.3) and (6.95,1.4) .. (10.93,3.29)   ;
\draw    (345.5,299) .. controls (321.37,320.54) and (244.84,315.17) .. (204.74,296.97) ;
\draw [shift={(202.93,296.13)}, rotate = 25.5] [color={rgb, 255:red, 0; green, 0; blue, 0 }  ][line width=0.75]    (10.93,-3.29) .. controls (6.95,-1.4) and (3.31,-0.3) .. (0,0) .. controls (3.31,0.3) and (6.95,1.4) .. (10.93,3.29)   ;
\draw   (375.66,72.73) .. controls (375.66,63.92) and (382.64,56.78) .. (391.24,56.78) .. controls (399.85,56.78) and (406.83,63.92) .. (406.83,72.73) .. controls (406.83,81.55) and (399.85,88.69) .. (391.24,88.69) .. controls (382.64,88.69) and (375.66,81.55) .. (375.66,72.73) -- cycle ;
\draw    (418.86,72.28) -- (449.82,72.28) ;
\draw [shift={(451.82,72.28)}, rotate = 180] [color={rgb, 255:red, 0; green, 0; blue, 0 }  ][line width=0.75]    (10.93,-3.29) .. controls (6.95,-1.4) and (3.31,-0.3) .. (0,0) .. controls (3.31,0.3) and (6.95,1.4) .. (10.93,3.29)   ;
\draw   (460.28,73.65) .. controls (460.28,64.83) and (467.42,57.69) .. (476.23,57.69) .. controls (485.04,57.69) and (492.19,64.83) .. (492.19,73.65) .. controls (492.19,82.46) and (485.04,89.6) .. (476.23,89.6) .. controls (467.42,89.6) and (460.28,82.46) .. (460.28,73.65) -- cycle ;
\draw    (505.27,72.28) -- (536.22,72.28) ;
\draw [shift={(538.22,72.28)}, rotate = 180] [color={rgb, 255:red, 0; green, 0; blue, 0 }  ][line width=0.75]    (10.93,-3.29) .. controls (6.95,-1.4) and (3.31,-0.3) .. (0,0) .. controls (3.31,0.3) and (6.95,1.4) .. (10.93,3.29)   ;
\draw   (546.69,73.65) .. controls (546.69,64.83) and (553.83,57.69) .. (562.64,57.69) .. controls (571.45,57.69) and (578.59,64.83) .. (578.59,73.65) .. controls (578.59,82.46) and (571.45,89.6) .. (562.64,89.6) .. controls (553.83,89.6) and (546.69,82.46) .. (546.69,73.65) -- cycle ;

\draw (35.2,210.23) node [anchor=north west][inner sep=0.75pt]   [align=left] {0};
\draw (117.6,211.15) node [anchor=north west][inner sep=0.75pt]   [align=left] {21};
\draw (201.85,209.23) node [anchor=north west][inner sep=0.75pt]   [align=left] {28};
\draw (289.85,210.23) node [anchor=north west][inner sep=0.75pt]   [align=left] {9};
\draw (376.73,212.15) node [anchor=north west][inner sep=0.75pt]   [align=left] {3};
\draw (464.92,211.15) node [anchor=north west][inner sep=0.75pt]   [align=left] {1};
\draw (342.49,127.99) node [anchor=north west][inner sep=0.75pt]   [align=left] {24};
\draw (427.44,126.08) node [anchor=north west][inner sep=0.75pt]   [align=left] {8};
\draw (382.59,65.37) node [anchor=north west][inner sep=0.75pt]   [align=left] {29};
\draw (188.85,264.23) node [anchor=north west][inner sep=0.75pt]   [align=left] {7};
\draw (265.85,265.23) node [anchor=north west][inner sep=0.75pt]   [align=left] {2};
\draw (347.73,270.15) node [anchor=north west][inner sep=0.75pt]   [align=left] {22};
\draw (467.73,66.15) node [anchor=north west][inner sep=0.75pt]   [align=left] {31};
\draw (554.92,65.15) node [anchor=north west][inner sep=0.75pt]   [align=left] {10};

\end{tikzpicture}

\caption{The subgraph of the multiplication transducer for $m=64$ that encodes all representations of $64$ as a quotient of sums of distinct powers of 3.}\label{digraph for 64}
\end{figure}

We begin by applying the algorithm as usual, stepping down from 9, and this yields $64=\frac{3^5+3^2+3+1}{3+1}$.
Letting $p(x)=x^5+x^2+x+1$ and $q(x)=x+1$, we have $64=\frac{p(3)}{q(3)}$.  Since $p(x)=\left(x^4-x^3+x^2+1\right)q(x)$, this is a universal representation of $64$.  Since $q(x)\neq1$, this universal representation is nontrivial.

Suppose we continued with the algorithm beyond $28$, the first carry value that is a sum of distinct powers of 3.  Would we get local representations or universal representations?  The first choice to be made occurs when the carry value is 9.  We will show that any representation of $64$ that takes a step up from carrying 9 to carrying 24 upon first arriving at the carry value 9 must be a local representation.  This is true for both decomposable and indecomposable representations, and we prove it using the same approach as in the proof of Theorem \ref{22allhard}.

\begin{theorem}
Any representation of $64$ that steps up upon first arriving at the carry value 9 is a local representation.  This holds for both decomposable and indecomposable representations.
\end{theorem}

\begin{proof}
Suppose to the contrary that $64 = \frac{p(3)}{q(3)}$ for some $p(x),q(x)\in\mathcal{P}$, $g(x):=\frac{p(x)}{q(x)}\in\mathbb{Z}[x]$, and the representation $\frac{p(3)}{q(3)}$ results from stepping up upon first arriving at the carry value 9. Let $r=\deg p(x)$ and $s=\deg q(x)$.  By Theorem \ref{T:feasible},
\[
\frac 23 \cdot 3^{r-s} < 64 < \frac 32 \cdot 3^{r-s},
\]
which implies $r -s =
\deg g(x) = 4$. Since the constant terms and the coefficients of $x^r$ in $p(x)$ and
$x^s$ in $q(x)$ are 1, we may then write $g(x)=x^4+e_3x^3+e_2x^2+e_1x+1$. It is now convenient to take $p(x) = 1 + \sum\limits_{i=1}^r  c_i x^i$ and $q(x) = 1 + \sum\limits_{i=1}^s d_i x^i$, where $c_i, d_i \in \{0,1\}$.

Since $64=g(3)=81+27e_3+9e_2+3e_1+1$, we have $-6=9e_3+3e_2+e_1$, and $e_1\equiv0\pmod 3$.  Since $e_1=c_1-d_1$, we see $c_1=d_1$ and $e_1=0$.  After the first step of the algorithm, we have a carry value of $21 \equiv 0 \pmod 3$ and thus must make a choice.  As in the case of $m=100$, one step, the step from $21$ down to 7, does not lead to any representations. Thus, to have a chance at actually obtaining a representation of $64$, we must step up from $21$ to $28$, and this corresponds to taking $c_1=d_1=1$.  Note that $e_1=0$ gives $-6=9e_3+3e_2$, so $-2=3e_3+e_2$, and $e_2\equiv1\pmod 3$.  Since $e_2=c_2-d_2-d_1e_1=c_2-d_2$, we see $c_2=1, d_2=0$, and $e_2=1$.  Then $-3=3e_3$, so $e_3=-1$.  Then $c_3=d_3+d_2e_1+d_1e_2+e_3$ is $c_3=d_3$.  Since we are interested in representations that step up from 9 to 24, we take $c_3=d_3=1$.  Then, substituting the values we have already determined, $c_4=d_4+d_3e_1+d_2e_2+d_1e_3+1$ becomes $c_4=d_4$.

Suppose $c_4=d_4=0$.  Then $c_5=d_5+d_4e_1+d_3e_2+d_2e_3+d_1$ is $c_5=d_5+2$, which is a contradiction to $c_5, d_5\in\{0,1\}$.

Now suppose $c_4=d_4=1$.  Then $c_5=d_5+d_4e_1+d_3e_2+d_2e_3+d_1$ is $c_5=d_5+2$, which is a contradiction to $c_5, d_5\in\{0,1\}$.

Since in all cases we arrive at a contradiction, we may conclude that $q(x)\nmid p(x)$, and thus all representations of $64$ that choose the step up when first carrying 9 are local.
\end{proof}

We conclude this section by discussing a specific example of a local representation of 64, one that steps up upon first arriving at the carry value $9$.

\begin{example}Suppose that $m=64=\frac{p}{q}\in\mathcal{A}/\mathcal{A}$.  As usual, we start with a carry value of 0, and we choose to add $m=64$ to the carry value, record $1\cdot3^0$ into $q$ and into $p$, then subtract 1 from the carry value, and move on to Step 1.  The remaining unrecorded value is $(64-1)\cdot3^0=63\cdot3^0=21\cdot3^1$, so the carry value at the end of Step 1 is now $21\equiv0\pmod{3}$.  As in the case of $m=100$, we must make a choice.  One choice, the step from $21$ down to 7, does not lead to any representations. Thus, to have a chance at actually obtaining a representation of $64$, we must step up from $21$ to $28$.  To implement this choice, we add $m=64$ to the carry value again, record $1\cdot3^1$ into $q$ and into $p$, and then subtract 1 from the carry value.  The remaining unrecorded value is $(21+64-1)\cdot3^1=84\cdot3^1=28\cdot3^2$, so the new carry value at the end of Step 2 is 28.  Since $28=3^3+3^0$ is a sum of distinct powers of 3, the algorithm would terminate here, and we would append the digits of the base 3 representation $28=[1001]_3$ to the left of $p$ to get the nontrivial universal representation $64=\frac{[100111]_3}{[11]_3}=\frac{3^5+3^2+3+1}{3+1}=\frac{256}{4}$ discussed above.  

At present, we seek to characterize representations of $64$ that step up upon first arriving at the carry value $9$.  To obtain such a representation of 64, we must continue the algorithm instead of allowing the algorithm to terminate in Step 2, when our carry value is $28\equiv1\pmod 3$.  To continue, we record $1\cdot3^2$ into $p$, subtract 1 from the carry value, and arrive at the unrecorded value $(28-1)\cdot3^2=27\cdot3^2=9\cdot3^3$.  The new carry value at the end of Step 3 is 9.  Because we are exploring representations of $64$ that step up upon first arriving at the carry value $9$, we choose to add $m=64$ to the carry value and record $1\cdot3^3$ into $q$ and into $p$.  Then, after subtracting 1 from the carry value, the remaining unrecorded value is $(9+64-1)\cdot3^3=72\cdot3^3=24\cdot3^4$, so our carry value at the end of Step 4 is 24.  We choose again to add $m=64$ to the carry value and record $1\cdot3^4$ into $q$ and into $p$.  After subtracting 1 from the carry value, we have the remaining unrecorded value $(24+64-1)\cdot3^4=87\cdot3^4=29\cdot3^5$, so we are left with a new carry value of 29 at the end of Step 5.  Since $29\equiv2\pmod{3}$, we must add $m=64$ to the carry value and correspondingly record $1\cdot3^5$ into $q$.  Then the remaining unrecorded value is $(29+64)\cdot3^5=93\cdot3^5=31\cdot3^6$, so our carry value at the end of Step 6 is $31$.  Since $31=3^3+3^1+3^0$ is a sum of distinct powers of 3, we can terminate this extension of the algorithm by appending the digits of the base 3 representation $31=[1011]_3$ to the left of $p$.  This yields the local representation $$64=\frac{[1011011111]_3}{[111011]_3}=\frac{3^9+3^7+3^6+3^4+3^3+3^2+3+1}{3^5+3^4+3^3+3+1}=\frac{22720}{355}.$$  Notice that this representation is $64=\frac{p(3)}{q(3)}$ where $p(x)=x^9+x^7+x^6+x^4+x^3+x^2+x+1$ and $q(x)=x^5+x^4+x^3+x+1$, and that $q(x)\nmid p(x)$.
\end{example}


\section{Catalog of knowledge}\label{catalog}
This table classifies all indecomposable representations of integers $m\equiv 1\pmod 3$ in the intervals given in \ref{feasible} up to $\frac{3}{2}\cdot 3^5$.  For some integers, such as 22, 34, and 229, information is also known about decomposable representations, as discussed for $22$ in Theorem \ref{22allhard} and for $34$ in Corollary \ref{34allhard}. In the table, $\bigstar$ indicates that $m$ has representations of that type, an empty cell indicates $m$ does not have representations of that type, and $?$ indicates that it is not known whether $m$ has representations of that type.  The only $?$ appears for local representations of $289$.  

In addition to the techniques already discussed in this paper, another approach we employed to find representations of an integer $m$ utilized 
Mathematica as follows. For all Newman polynomials $p(x)$ up to a fixed degree, construct a table of the values $p(3)$. It was computationally feasible for us to do this up to degree $18$.  Search for those values $p(3)$ which are multiples of $m$. If $p(3) = m\cdot t$, then see if the ternary expansion of $t$ uses only the digits 0 and 1.
If so, then $t = q(3)$ for a Newman polynomial $q(x)$. Finally, check if $q(x)\mid p(x)$.

\newpage

\begin{center}
\renewcommand{\arraystretch}{.775}
\begin{minipage}{.3\textwidth}
\begin{center}
    \begin{tabular}{|r|c|c|}
    \hline
       m&  universal& local\\      
       \hline
       1 & $\bigstar$ & \\
       \hline
       4 & $\bigstar$ & \\
       \hline
       7 & $\bigstar$ & \\
       \hline
          10 & $\bigstar$ & \\
       \hline
       13 & $\bigstar$ &  \\
       \hline
       19 & $\bigstar$ &  \\
       \hline
       22 & & $\bigstar$ \\
       \hline
       25 & $\bigstar$ & 
        \\
       \hline
       28  & $\bigstar$ & \\
       \hline
       31 &$\bigstar$  & $\bigstar$  \\
       \hline
       34   & & $\bigstar$ \\
       \hline
       37 &$\bigstar$  & $\bigstar$  \\
       \hline
       40 & $\bigstar$ &  \\
       \hline
       55 & $\bigstar$ &  \\
       \hline
       58 & &  $\bigstar$  \\
       \hline
       61 & $\bigstar$ &  \\
       \hline
       64 & $\bigstar$ & $\bigstar$  \\
       \hline
       67 &  & $\bigstar$ \\
       \hline
       70 & $\bigstar$ & $\bigstar$  \\
       \hline
       73 & $\bigstar$ &  \\
       \hline
       76  &$\bigstar$ & $\bigstar$  \\
       \hline
       79 & $\bigstar$ &  \\
       \hline
       82 & $\bigstar$ &  \\
       \hline
       85 & $\bigstar$ & $\bigstar$  \\
       \hline
       88 &$\bigstar$ & $\bigstar$  \\
       \hline 
       91 & $\bigstar$ & \\
       \hline
       94 &$\bigstar$  & $\bigstar$   \\
       \hline
       97  &  & $\bigstar$ \\
       \hline
       100 & $\bigstar$ &  \\
       \hline
       103  &  & $\bigstar$ \\
       \hline
       106 &  & $\bigstar$  \\
       \hline
       109  &$\bigstar$ & $\bigstar$ \\
       \hline
       112 &$\bigstar$  & $\bigstar$  \\
       \hline
       115 &  & $\bigstar$  \\
       \hline
       118 &$\bigstar$ & $\bigstar$ \\
       \hline
       121 & $\bigstar$ &   \\
       \hline
       
                    \end{tabular}
                    \end{center}
                    \vspace{4.65cm}
                    \end{minipage}%
\begin{minipage}{.3\textwidth}
\begin{center}
    \begin{tabular}{|r|c|c|}
    \hline
       m&  universal& local\\      
       \hline
        163 & $\bigstar$ &  \\
       \hline
	166 &  & $\bigstar$  \\
       \hline
       169 & $\bigstar$ & $\bigstar$ \\
       \hline
       172 &  & $\bigstar$ \\
       \hline
       175 & & $\bigstar$  \\
       \hline
       178 & & $\bigstar$ \\
       \hline
       181 & $\bigstar$ &  \\
       \hline
       184 & & $\bigstar$  \\
       \hline
       187 &$\bigstar$ & $\bigstar$  \\
       \hline
       190 &$\bigstar$ &$\bigstar$\\
       \hline
       193 & & $\bigstar$ \\
       \hline
       196 &$\bigstar$ & $\bigstar$ \\
       \hline
       199 & & $\bigstar$ \\
       \hline
       202 & & $\bigstar$  \\
       \hline
       205 & & $\bigstar$  \\
       \hline
       208 & & $\bigstar$ \\
       \hline
       211 & $\bigstar$ & $\bigstar$ \\
       \hline
       214 & & $\bigstar$ \\
       \hline
       217& $\bigstar$ &  \\
       \hline
       220&$\bigstar$ & $\bigstar$  \\
       \hline
       223 &$\bigstar$ & $\bigstar$ \\
       \hline
       226 &  & $\bigstar$  \\
       \hline
       229 & & $\bigstar$ \\
       \hline
       232 & & $\bigstar$ \\
       \hline
       235 & $\bigstar$ & \\
       \hline
       238 & & $\bigstar$ \\
       \hline
       241& $\bigstar$ &  \\
       \hline
       244 & $\bigstar$ & \\
       \hline
       247 &$\bigstar$ & $\bigstar$  \\
       \hline
       250 &$\bigstar$ & $\bigstar$ \\
       \hline
       253 &$\bigstar$ & $\bigstar$ \\
       \hline
       256 &$\bigstar$ & $\bigstar$ \\
       \hline
       259 & & $\bigstar$ \\
       \hline
       262&  & $\bigstar$  \\
       \hline
       265 & & $\bigstar$ \\
       \hline
       268 &  & $\bigstar$  \\
       \hline
       271&$\bigstar$ & $\bigstar$  \\
       \hline
       274& $\bigstar$ & $\bigstar$  \\
       \hline
       277 & & $\bigstar$ \\
       \hline
       280 &$\bigstar$ & $\bigstar$  \\
       \hline
       283 &$\bigstar$ & $\bigstar$   \\
       \hline
       286 & & $\bigstar$ \\
       \hline
       289 & $\bigstar$ & $?$  \\
       \hline
       292 & $\bigstar$ & $\bigstar$  \\
       \hline
       295 & & $\bigstar$ \\
       \hline
       298 & & $\bigstar$ \\
       \hline
       
   \end{tabular}
   \end{center}
   \end{minipage}%
	\begin{minipage}{.3\textwidth}	
	\begin{center}
    \begin{tabular}{|r|c|c|}
    \hline
       m&  universal& local\\  
       \hline 
     301  &$\bigstar$ & $\bigstar$ \\
       \hline
304 & & $\bigstar$\\
       \hline

307 & & $\bigstar$\\
       \hline
310 & & $\bigstar$\\
       \hline
313 & & $\bigstar$\\
       \hline
316&$\bigstar$ & $\bigstar$\\
       \hline
319 &$\bigstar$ & $\bigstar$\\
       \hline
322 &  & $\bigstar$\\
       \hline
325&$\bigstar$ & $\bigstar$\\
       \hline
328&$\bigstar$ & $\bigstar$\\
       \hline
331& & $\bigstar$\\
       \hline
334&$\bigstar$ & $\bigstar$\\
       \hline
337& $\bigstar$ & $\bigstar$\\
       \hline
340& & $\bigstar$\\
\hline
343&  &$\bigstar$\\
\hline
346&  &$\bigstar$\\
\hline
349 &  &$\bigstar$\\
\hline
352 & $\bigstar$ &$\bigstar$\\
\hline
355 & $\bigstar$ &$\bigstar$\\
\hline
358 &  &$\bigstar$\\
\hline
361 & $\bigstar$ &$\bigstar$\\
\hline
364 &$\bigstar$ &\\
       \hline

   \end{tabular}
   \end{center}
   \vspace{11.26cm}
	\end{minipage}%
	\end{center}

\section{Further directions}\label{FutureWork}
Ultimately, we want to know which integers can be written as a quotient of sums of distinct powers of 3.  An additional long-term goal is to find a complete classification of the types of representations
integers can have as quotients of sums of distinct powers of 3. For example, we seek a complete description of the 
integers congruent to 1 mod 3 in \eqref{feasible} which do not lie in $\mathcal A/\mathcal A$. This paper makes progress toward this goal and also provides a launching point for several related questions and directions. 

\subsection{Properties of the Directed Graphs}
In Figure \ref{subgraphoftransducer} in Section \ref{sec:algorithm}, we introduce an example of an expanded digraph generated as a subgraph of the multiplication transducer that performs multiplication by $m$ in base 3. These digraphs are worthy of study in their own right, beyond the information they encode regarding representations of integers as quotients of sums of distinct powers of 3. For any positive integer $m=3t+1$, define the digraph $D_m$ as follows:
The vertices of $D_m$ are a subset of $\{0,\dots,\lfloor m/2 \rfloor\}$. There are four
kinds of directed edges in $D_m$.  What kind of edge we see originating at a given vertex depends on the congruence class of the given vertex modulo 3.  The four kinds of edges are
\[
\{3k\rightarrow k, \quad3k\rightarrow k+t,\quad3k+1\rightarrow k, \quad3k+2\rightarrow k+t+1\}.
\]
If $i$ is in the vertex set of $D_m$ and $i \rightarrow j$ is in the set of possible edges above, then $j$ must also belong to the vertex set of $D_m$.

This digraph $D_m$ is the digraph produced by applying the algorithm of Section \ref{sec:algorithm} to $m$, as done in Figure \ref{subgraphoftransducer} for $m=22$, in Figure \ref{digraph for 100} for $m=100$, and in Figure \ref{digraph for 64} for $m=64$,  and the four kinds of edges correspond to the four possible ways to move from Step $i$ to Step $i+1$ in the algorithm.
Every representation of $m$ as $\frac{p(3)}{q(3)}$ with $p(x),q(x) \in \mathcal P$ can be 
read off from the vertices in a walk from the vertex 0 to itself. When no such walk exists, as
in the case $m=529$, there is no such representation. More information about these digraphs and various related topics will appear in \cite{ADRS}.

\subsection{Extending to different bases}
It is natural to ask the questions of this paper for digital bases $b \neq 3$.  If $b=2$, then the standard binary representation for $n$, $n = \sum_i a_i2^i$ with $a_i \in \{0,1\}$, gives $n = p(2)$, where $p(x) =  \sum_i a_ix^i \in \mathcal P$.

For any $b\ge 2$, local representations still exist: if $p(x),q(x) \in \mathcal P$ and
$p(x) = q(x)t(x)$ for $t(x) \in \mathbb{Z}[x]$, then $t(b) = \frac{p(b)}{q(b)}$ will
be a quotient of sums of distinct powers of $b$. For example, 
\[
b^2 - b + 1 = [(b-1)\,1]_b = \frac{b^3+1}{b+1}.
\]
On the other hand, $b^2 \equiv 1 \pmod{b+1}$, and
\begin{equation}\label{b's}
\frac{b^{2b} + \cdots + b^4 + b^2 + 1}{b+1} \in \mathbb Z.
\end{equation}
Further, $ (x +1) \nmid \left(x^{2b} + \cdots + x^4 + x^2 + 1\right)$, because $x=-1$ is not a root
of the polynomial on the right. Thus, for every base $b$, there are local representations 
of quotients of distinct sums of powers of $b$; for $b=3$, \eqref{b's} is a local 
representation of $\frac{729+81+9+1}{3+1} = 205$. 

The phrase ``quotients of sums of distinct powers of $b$" may be rewritten as ``quotients of integers with base $b$ representations using only digits from $\{0, 1\}$." Written this way, we can extend our work to ask similar questions for different sets of digits when $b\geq 3$. The fourth-named author studied generalizations of this problem in her dissertation \cite{SST}, exploring quotients of integers which, when written in base $b$, have a restricted digit set.

\section{Acknowledgments}
The second author was supported by an AMS-Simons Travel Grant.  The third and fourth authors wish to acknowledge the warm hospitality of the Department of Mathematics of the University of Texas at Tyler during a visit in June 2022.


\end{document}